\documentclass[12pt,a4paper]{amsart}
\usepackage{amsmath,amssymb,amsthm,mathtools}
\usepackage[margin=2.7cm]{geometry}
\usepackage[shortlabels]{enumitem}
\usepackage[colorlinks=true,linkcolor=red,citecolor=blue]{hyperref}

\DeclareMathOperator{\supp}{supp}
\newcommand{\R}{\mathbb{R}}
\newcommand{\ind}[1]{\mathbf{1}_{#1}}

\newtheorem{theorem}{Theorem}[section]
\newtheorem{lemma}[theorem]{Lemma}
\newtheorem{corollary}[theorem]{Corollary}
\newtheorem{proposition}[theorem]{Proposition}
\theoremstyle{definition}
\newtheorem{definition}[theorem]{Definition}
\newtheorem{example}[theorem]{Example}
\theoremstyle{remark}
\newtheorem{remark}[theorem]{Remark}

\numberwithin{equation}{section}
\makeatletter
\def\l@subsection{\@tocline{2}{0pt}{2.5pc}{5pc}{}}
\makeatother

\begin{document}

\title[Packing intermediate dimensions]{Packing intermediate dimensions of measures and applications}

\author{Nicolas Angelini}
\address{Department of Mathematics, Universit\`a degli Studi di Trento. Trento, Italy.}
\address{Instituto de Matem\'atica Aplicada San Luis (IMASL), CONICET, Argentina.
Departamento de Matem\'atica, FCFMyN, Universidad Nacional de San Luis, Argentina.}
\email{nicolas.angelini.2015@gmail.com}


\date{\today}

\begin{abstract}
We introduce a spectrum of dimensions of measures, the \emph{packing
intermediate dimensions} $\dim_{P,\theta}\mu$, $\theta\in(0,1]$, of a Borel
probability measure $\mu$ on $\mathbb{R}^n$ and the \emph{quasi-Hausdorff dimension}
$\dim_{qH}\mu$. We show that this spectrum encodes the information on
$\mu$ that governs the packing dimension of its projections and
H\"{o}lder images. First, $\dim_{qH}\mu$ is the exact threshold for the
packing dimension profiles of Falconer and Howroyd: for every real
$0<s\leq n$, $\dim_P^s\mu=s$ if and only if $s\leq\dim_{qH}\mu$. Second,
we prove a lower bound for the profiles in terms of $\dim_H\mu$,
$\dim_{P,\theta}\mu$ and the Assouad dimension of $\operatorname{supp}\mu$,
which improves the bound of Falconer and Mattila, the best previously
known general lower bound for the packing dimension of projected
measures. The improvement is strict whenever $\dim_H\mu<\dim_P\mu$ and
$\dim_A(\operatorname{supp}\mu)<n$, and when $\dim_A(\operatorname{supp}\mu)=n$
it occurs exactly when $\dim_{P,\theta}\mu$ exceeds an explicit universal
lower bound for some $\theta$ in an explicit range. We present examples of measures
with $\dim_H\mu<\dim_P\mu<\dim_A(\operatorname{supp}\mu)=n$ for which our
bound gives the exact value $\dim_P\mu_V=\min\{\dim_P\mu,m\}$ for almost
every $m$-dimensional subspace $V$, whereas the bound of Falconer and
Mattila is strictly smaller. We also deduce
the estimate
$\dim_P^s\mu\geq\dim_P\mu-\max\{0,\dim_A(\operatorname{supp}\mu)-s\}$,
which recovers a projection theorem of Falconer, Fraser and Shmerkin for
sets. As applications, for a large family of random $\alpha$-H\"{o}lder
maps $f_\omega:\mathbb{R}^n\to\mathbb{R}^m$, including index-$\alpha$
fractional Brownian motion and orthogonal projections, we prove that, for
compactly supported $\mu$, $\dim_P\mu_{f_\omega}=m$ almost surely if and
only if $\alpha m\leq\dim_{qH}\mu$, and that the packing analogue of
Kahane's formula holds for fractional Brownian motions of every index and
every target dimension if and only if $\dim_{qH}\mu=\dim_P\mu$. Analogous
results are obtained for analytic sets.
\end{abstract}

\maketitle

\setcounter{tocdepth}{2}
\tableofcontents

\newpage

\section{Introduction}

Packing dimension $\dim_P$ was introduced as a dual concept to Hausdorff
dimension and has since been studied extensively, both for sets and
measures.
A fundamental problem in fractal geometry is understanding how dimension
behaves under mappings: given a measure $\mu$ and a map $f$, how does
$\dim\mu_f$ relate to $\dim\mu$, where $\dim$ denotes any notion of
dimension and $\mu_f$ denotes the image of $\mu$ under $f$? A natural
class of examples is provided by $\alpha$-H\"{o}lder maps, which includes
orthogonal projections (the case $\alpha=1$) and fractional Brownian
motion.

For Hausdorff dimension the answer is classical. If $\mu$ is a finite
Borel measure on $\R^n$ and $\mu_V$ denotes the orthogonal projection of
$\mu$ onto an $m$-dimensional subspace $V\in G(n,m)$, then
\begin{equation}\label{eq:intro_HT}
\dim_H\mu_V=\min\{\dim_H\mu,m\}\qquad\text{for }\gamma_{n,m}\text{-almost every }V\in G(n,m),
\end{equation}
see \cite{hu1994fractal}. For an index-$\alpha$ fractional Brownian motion
$X:\R^n\to\R^m$ and a compactly supported finite Borel measure $\mu$ on
$\R^n$, Kahane \cite{kahane1985some} proved the set version of the formula
\begin{equation}\label{eq:kahane}
\dim_H\mu_X=\min\Bigl\{m,\,\frac{\dim_H\mu}{\alpha}\Bigr\}\qquad\text{almost surely},
\end{equation}
from which the measure version follows.

Unfortunately, packing dimension does not admit analogous formulas in
general. Falconer and Howroyd \cite{falconer1997packing} introduced the
packing dimension profiles $\dim_P^{s}\mu$, $s>0$ (see
Definition~\ref{def:packing_profile}), and proved that for every finite
Borel measure $\mu$ on $\R^n$,
\begin{equation}\label{eq:intro_FH}
\dim_P\mu_V=\dim_P^{m}\mu\qquad\text{for }\gamma_{n,m}\text{-almost every }V\in G(n,m),
\end{equation}
while Xiao \cite[Theorem 3.1]{xiao1997packing} proved that, for the
index-$\alpha$ fractional Brownian motion $X:\R^n\to\R^m$,
\begin{equation}\label{eq:intro_xiao}
\dim_P\mu_X=\frac{1}{\alpha}\,\dim_P^{\alpha m}\mu\qquad\text{almost surely}.
\end{equation}
These are complete formulas for the packing dimension of the images,
expressed in terms of profiles that encode multiscale geometric
information about $\mu$. However, such profiles are difficult to compute
in practice, and a natural question is: under which conditions can one
recover formulas analogous to \eqref{eq:intro_HT} and \eqref{eq:kahane}?
It is straightforward to see that $\dim_P\mu_V\leq\min\{\dim_P\mu,m\}$
for every $V\in G(n,m)$. In the other direction, Falconer and Mattila
\cite[Theorem 3.3]{falconer1996packing} proved that if $\mu$ is a Borel
probability measure on $\R^n$, then for $\gamma_{n,m}$-almost every $V$,
\begin{equation}\label{eq:intro_FM}
\dim_P\mu_V\ \geq\
\begin{cases}
\dfrac{\dim_P\mu\left(1-\frac{1}{n}\dim_H\mu\right)}
{1+\left(\frac{1}{m}-\frac{1}{n}\right)\dim_P\mu-\frac{1}{m}\dim_H\mu},
& \text{if } \dim_H\mu\leq m, \\[1.2em]
m, & \text{if } \dim_H\mu\geq m,
\end{cases}
\end{equation}
and they showed that this bound is sharp in general. For sets, Falconer,
Fraser and Shmerkin \cite{falconer2021assouad} showed that the possible
drop of the packing dimension under projections is controlled by the
Assouad dimension of the set.\\

\noindent
\textbf{Packing intermediate dimensions}

In recent years several families of dimensions interpolating between two
classical notions have been introduced, such as the intermediate
dimensions of Falconer, Fraser and Kempton \cite{falconer2020intermediate},
which interpolate between the Hausdorff and box dimensions, or the
Assouad spectrum, which interpolates between the box and the Assouad
dimensions \cite{fraser2018new}. In this article we introduce a family
of this kind for the packing dimension of measures, the \emph{packing
intermediate dimensions} $\dim_{P,\theta}\mu$, $\theta\in(0,1]$ (see
Definition~\ref{def:pid}), and we show that they provide a natural
framework to address the question above. They are defined through the
kernels
$$\phi_{r,\theta}^{t,n}(z)=\min\bigl\{1,\ r^{t}|z|^{-t},\
r^{\theta(n-t)+t}|z|^{-n}\bigr\},$$
which coincide with the kernel $\min\{1,r^t|z|^{-t}\}$ of the
$t$-dimensional profile for $|z|\leq r^{\theta}$, and decay like the
kernel of the $n$-dimensional profile for $|z|>r^{\theta}$. Equivalently
(Theorem~\ref{thm:windows}), $\dim_{P,\theta}\mu$ is the supremum of the
exponents $t\geq 0$ such that, at $\mu$-almost every point $x$ and for
every $\varepsilon>0$, there are arbitrarily small windows of scales
$[\rho,\rho^{\theta}]$ along which $\mu(B(x,r))\leq\varepsilon r^{t}$. For
$\theta=1$ the window reduces to a single scale and one recovers the
packing dimension of $\mu$, moreover at $\theta=0$ in its equivalent definition (Theorem \ref{thm:windows}) it reduces to Hausdorff dimension. The map $\theta\mapsto\dim_{P,\theta}\mu$ is continuous in $(0,1]$ and 
nondecreasing, and
$$\dim_H\mu\leq\dim_{P,\theta}\mu\leq\dim_P\mu\qquad\text{for every }\theta\in(0,1].$$
We call the limit
$$\dim_{qH}\mu:=\lim_{\theta\to0}\dim_{P,\theta}\mu$$
the \emph{quasi-Hausdorff dimension} of $\mu$. By its construction it is
of the same nature as the quasi-Assouad dimension, which is the limit of
the Assouad spectrum as the parameter tends to the endpoint where the
Assouad dimension sits \cite{fraser2018new}; it is also the packing
counterpart of the quasi-Hausdorff dimensions of sets studied in
\cite{angelini2025critical} for the intermediate dimensions introduced in \cite{falconer2020intermediate}. In general
$\dim_H\mu\leq\dim_{qH}\mu\leq\dim_P\mu$, and both inequalities can be
strict. A useful tool that we prove is the following
universal lower bound (Lemma~\ref{lem:floor}): writing $P=\dim_P\mu$ and
$N=\dim_A(\supp\mu)$ for compactly supported measures,
$$\dim_{P,\theta}\mu\ \geq\ \frac{\theta NP}{N-(1-\theta)P}\qquad\text{for every }\theta\in(0,1].$$
In particular, $\dim_{P,\theta}\mu=\dim_{qH}\mu=\dim_P\mu$ whenever
$\dim_P\mu=\dim_A(\supp\mu)$. On the other hand, for the measures
constructed by Falconer and Mattila to show the sharpness of
\eqref{eq:intro_FM}, this lower bound is attained and
$\dim_{qH}\mu=\dim_H\mu<\dim_P\mu$ (Example~\ref{ex:FM}).

The characterization through windows of scales is the packing analogue
of the definition of the intermediate dimensions of measures introduced
in \cite{ANGELINI2026130039}, which interpolate between the Hausdorff and
Minkowski dimensions; the difference lies in the quantifiers, since here
the scales are allowed to depend on the point (see
Remark~\ref{rem:windows}). Results on the behaviour of intermediate
dimensions under projections and other images were obtained in
\cite{angelini2025critical}, see also \cite{burrell2021projection, Burrell}.\\

\noindent
\textbf{Main results}

Our first main result shows that the quasi-Hausdorff dimension is the
exact threshold at which the packing dimension profiles attain their
maximal value. More precisely, for every $\theta\in(0,1]$ and every real
number $0<s\leq n$,
\begin{equation}\label{eq:intro_two_sided}
(1-\theta)\min\{s,\dim_{P,\theta}\mu\}\ \leq\ \dim_P^{s}\mu\ \leq\
(1-\theta)\,s+\theta\min\{s,\dim_{P,\theta}\mu\},
\end{equation}
and consequently
\begin{equation}\label{eq:intro_threshold}
\dim_P^{s}\mu=s\quad\Longleftrightarrow\quad s\leq\dim_{qH}\mu;
\end{equation}
see Theorem~\ref{thm:threshold}. Moreover, the packing analogue of
\eqref{eq:intro_HT} and \eqref{eq:kahane} at the level of profiles,
namely $\dim_P^{s}\mu=\min\{s,\dim_P\mu\}$ for every $s>0$, holds if and
only if $\dim_{qH}\mu=\dim_P\mu$ (Corollary~\ref{cor:kahane_profiles}); as
noted above, this is the case whenever $\dim_P\mu=\dim_A(\supp\mu)$.

Our second group of results concerns lower bounds for the profiles in the
complementary regime $s>\dim_{qH}\mu$. We show that the Assouad dimension
of the support controls the possible drop of the profiles: if $\supp\mu$
is compact, then
\begin{equation}\label{eq:intro_assouad}
\dim_P^{s}\mu\ \geq\ \dim_P\mu-\max\{0,\ \dim_A(\supp\mu)-s\}
\end{equation}
(Corollary~\ref{cor:assouad}). Note that the measure enters this bound
only through $\dim_P\mu$, so that the size of the possible loss is
governed by the support alone and is common to all measures supported on
a given set. Taking supremum over the measures supported on a set, one
recovers the corresponding result of Falconer, Fraser and Shmerkin
\cite{falconer2021assouad} for sets in its Assouad dimension form
(Remark~\ref{rem:FFS}). In fact, \eqref{eq:intro_assouad} is a
consequence of a sharper bound involving the Hausdorff dimension and the
packing intermediate dimensions. Writing $d_H=\dim_H\mu$,
$P_\theta=\dim_{P,\theta}\mu$ and $N=\dim_A(\supp\mu)$ (with $N=n$ if
$\supp\mu$ is not compact), we prove that for $d_H<s\leq n$,
\begin{equation}\label{eq:intro_FMimproved}
\dim_P^{s}\mu\ \geq\ \sup_{0<\theta\leq 1}\min\Bigl\{P_\theta,\
\frac{s\,P_\theta\,(N-d_H)}{P_\theta(N-d_H)+\theta(s-d_H)(N-P_\theta)}\Bigr\};
\end{equation}
see Theorem~\ref{thm:FMimproved}. For $s=m$, the term $\theta=1$ in
\eqref{eq:intro_FMimproved} with $N$ replaced by $n$ is exactly the bound
\eqref{eq:intro_FM} of Falconer and Mattila. The term $\theta=1$ alone
already improves \eqref{eq:intro_FM} strictly whenever
$\dim_H\mu<\dim_P\mu$ and $\dim_A(\supp\mu)<n$, and it implies
\eqref{eq:intro_assouad}. The terms $\theta<1$, which are only available
through the packing intermediate dimensions, improve it strictly, even
when $\dim_A(\supp\mu)=n$, as soon as $\dim_{P,\theta}\mu$ exceeds the
universal lower bound above for some $\theta$ in an explicit interval
$[\theta^\dagger,1)$ (Corollary~\ref{cor:floor_improvement} and example \ref{ex:constant-Rn}). Thus the
introduction of the packing intermediate dimensions leads to an
improvement of the bound of Falconer and Mattila, which was the best
known general lower bound for the packing dimension of projected
measures. For the measures that show the sharpness of
\eqref{eq:intro_FM}, the universal lower bound is attained and
\eqref{eq:intro_FMimproved} reduces to \eqref{eq:intro_FM}, in accordance
with its optimality.\\

In Section~\ref{sec:images} we apply the results above to a large family
of random $\alpha$-H\"{o}lder maps $f_\omega:\R^n\to\R^m$
(Definition~\ref{def:random_holder}), which includes index-$\alpha$
fractional Brownian motion and, for $\alpha=1$, orthogonal projections
onto $m$-dimensional subspaces as particular cases. For these maps, and
compactly supported measures, we show that
$$\dim_P\mu_{f_\omega}=\frac{1}{\alpha}\,\dim_P^{\alpha m}\mu\qquad\text{almost surely}$$
(Theorem~\ref{thm:images}), which contains \eqref{eq:intro_FH} and
\eqref{eq:intro_xiao} as particular cases, so that all the results above
apply to the images. In particular,
$$\dim_P\mu_{f_\omega}=m\ \text{ almost surely}\quad\Longleftrightarrow\quad
\alpha m\leq\dim_{qH}\mu,$$
the forward implication being valid for every (deterministic)
$\alpha$-H\"{o}lder map (Corollary~\ref{cor:images_threshold}). For
orthogonal projections this reads
$$\dim_P\mu_V=m\ \text{ for }\gamma_{n,m}\text{-almost every }V\in G(n,m)
\quad\Longleftrightarrow\quad m\leq\dim_{qH}\mu,$$
while \eqref{eq:intro_assouad} gives $\dim_P\mu_V\geq\dim_P\mu-\max\{0,\
\dim_A(\supp\mu)-m\}$ for almost every $V$; in particular, if
$E\subset\R^n$ is compact and $m\geq\dim_A E$, the packing dimension of
every measure supported on $E$ is preserved under orthogonal projection
onto almost every $m$-dimensional subspace
(Corollary~\ref{cor:projections}). Moreover,
Xiao's formula \eqref{eq:intro_xiao} reduces to the packing analogue of
Kahane's formula \eqref{eq:kahane}, namely
$\dim_P\mu_X=\min\{m,\alpha^{-1}\dim_P\mu\}$ almost surely for every index
$\alpha$ and every $m$, if and only if $\dim_{qH}\mu=\dim_P\mu$
(Corollary~\ref{cor:images_kahane}).

In section \ref{sec:ex} we present several examples regarding the optimality of our main results.

Finally, in Section~\ref{sec:sets} we obtain the analogous results for
analytic sets. Following \cite{xiao1997packing}, the packing dimension of
the image of an analytic set can be recovered from the measures it
supports, which allows us to transfer the threshold phenomenon from
measures to sets. Defining, for an analytic set $E\subset\R^n$,
$$\dim_{P,\theta}E:=\sup\{\dim_{P,\theta}\mu\,:\,\mu\in\mathcal{M}_c^+(E)\},$$
where $\mathcal{M}_c^+(E)$ denotes the family of finite Borel measures
with compact support contained in $E$, we prove that
$$\dim_P f_\omega(E)=m\ \text{ almost surely}\quad\Longleftrightarrow\quad
\alpha m\leq\lim_{\theta\to 0}\dim_{P,\theta}E$$
(Theorem~\ref{thm:sets_images}).

The paper is organized as follows. In Section~\ref{sec:prelim} we recall
the necessary background on dimensions of measures, packing dimension
profiles and fractional Brownian motion. Section~\ref{sec:growth}
contains the estimates for the growth of the measure of balls on which
most of our arguments rely; the key point is a sharpening of an estimate
of Falconer and Mattila, in which the ambient dimension is replaced by
the Assouad dimension of the support. In Section~\ref{sec:pid} we
introduce the packing intermediate dimensions and the quasi-Hausdorff
dimension and establish their basic properties. Section~\ref{sec:threshold}
contains the threshold characterization for the profiles, and
Section~\ref{sec:lower} the lower bounds. In Section~\ref{sec:images} we
apply these results to random H\"{o}lder maps, in \ref{sec:ex} we illustrate with examples and in
Section~\ref{sec:sets} we extend them to analytic sets.

\section{Preliminaries}\label{sec:prelim}

We write $B(x,r)$ for the closed ball of centre $x$ and radius $r$, and
for a bounded set $E\subset\R^n$ we write $|E|$ for its diameter.
Throughout the paper, $\mu$ denotes a Borel probability measure on
$\R^n$, that is $\mu(\R^n)=1$, and we write $\supp\mu$ for its support;
we do not assume that $\supp\mu$ is compact unless this is said
explicitly. Given a Borel set $A\subset\R^n$ with $\mu(\R^n\setminus A)=0$
and a Borel measurable function $f:A\to\R^m$, the image (or pushforward)
measure of $\mu$ under $f$ is
$$\mu_f(B):=\mu\bigl(f^{-1}(B)\bigr)=\mu\{x\in A:f(x)\in B\},
\qquad B\subset\R^m \text{ Borel}.$$
Given $0<\alpha\leq 1$, a function $f:A\subseteq\R^n\to\R^m$ is
\emph{$\alpha$-H\"{o}lder continuous} if there exists a constant $c>0$
such that
$$|f(x)-f(y)|\leq c\,|x-y|^{\alpha}\qquad\text{for all }x,y\in A.$$
For $1\leq m\leq n$, we write $G(n,m)$ for the Grassmannian of
$m$-dimensional linear subspaces of $\R^n$, endowed with its natural
invariant probability measure $\gamma_{n,m}$, see
\cite[Chapter 3]{mattila1999geometry}. For $V\in G(n,m)$, $\pi_V:\R^n\to V$
denotes the orthogonal projection onto $V$, and $\mu_V:=\mu\circ\pi_V^{-1}$
is the projected measure, a Borel probability measure on $V$. All the
dimensions considered in this paper are invariant under isometries, so
we may identify $V$ with $\R^m$ through a linear isometry whenever
convenient.\\

\noindent

The Assouad dimension of a nonempty set $E\subset\R^n$ is defined by
\begin{align*}
\dim_A E := \inf \Bigl\{ s \geq 0 :\ &\exists\, C>0 \text{ such that for all } 0<r<R \text{ and } x\in E,\\
&\mathcal{N}_r\bigl(E \cap B(x,R)\bigr)\leq C \left(\frac{R}{r}\right)^s\Bigr\},
\end{align*}
where $\mathcal{N}_r(A)$ denotes the minimal number of balls of radius
$r$ needed to cover $A$. The value of $\dim_A E$ is unchanged if the
covering balls are replaced by cubes of side length $r$, by the cubes of
the mesh of side length $r$ oriented at the origin that intersect $A$, or
by maximal $r$-separated subsets of $A$, and also if the ball $B(x,R)$ is
replaced by the cube $Q(x,R)$ of side length $R$ centred at $x$. We use
these alternative formulations interchangeably, and refer to
\cite{fraser2020assouad} for a detailed discussion. Given a Borel
probability measure $\mu$ on $\R^n$, we set
\begin{equation}\label{eq:Nmu}
N_\mu:=\begin{cases}\dim_A(\supp\mu), & \text{if } \supp\mu \text{ is compact},\\ n, & \text{otherwise}.\end{cases}
\end{equation}

\noindent
The Hausdorff and packing dimensions of a Borel probability measure $\mu$
on $\R^n$ are defined by
\begin{align*}
\dim_H\mu&=\inf\{\dim_H E : E \text{ is a Borel set with }\mu(E)>0\},\\
\dim_P\mu&=\inf\{\dim_P E : E \text{ is a Borel set with }\mu(E)>0\},
\end{align*}
where $\dim_H E$ and $\dim_P E$ denote the Hausdorff and packing
dimensions of the set $E$. They admit the following local
characterizations, see \cite{falconer1996packing, hu1994fractal}:
\[
\dim_H\mu=\sup\left\{t\geq 0 : \limsup_{r\to 0}\frac{\mu(B(x,r))}{r^{t}}=0
\ \text{ for }\mu\text{-a.e.\ }x\in\R^n\right\},
\]
and 
\begin{equation}\label{eq:hu}
\dim_P\mu=\sup\left\{t\geq 0 : \liminf_{r\to 0}\frac{\mu(B(x,r))}{r^{t}}=0
\ \text{ for }\mu\text{-a.e.\ }x\in\R^n\right\}.
\end{equation}
Note that these quantities do not change if $\mu$ is multiplied by a positive
constant. If $\supp\mu$ is compact, then
$\dim_P\mu\leq\dim_P(\supp\mu)\leq\dim_A(\supp\mu)\leq n$; hence, in all
cases,
\begin{equation}\label{eq:chain}
\dim_H\mu\leq\dim_P\mu\leq N_\mu\leq n .
\end{equation}

\noindent

Falconer and Howroyd \cite{falconer1997packing} introduced the packing
dimension profiles, which recover the packing dimension of the typical
projection of a measure onto a lower-dimensional subspace.

\begin{definition}\label{def:packing_profile}
For a finite Borel measure $\mu$ on $\R^n$ and $s>0$, set
$$F_s^\mu(x,r)=\int_{\R^n}\min\{1,\,r^s|x-y|^{-s}\}\,d\mu(y),$$
and define the \emph{$s$-dimensional packing dimension profile} of $\mu$
by
$$\dim_P^s\mu=\sup\left\{t\geq 0 : \liminf_{r\to 0}\frac{F_s^{\mu}(x,r)}{r^{t}}=0
\ \text{ for }\mu\text{-a.e.\ }x\in\R^n\right\}.$$
\end{definition}

At the ambient exponent the profile reduces to the packing dimension: by
\cite[Corollary 3]{falconer1997packing},
\begin{equation}\label{eq:FHcor3}
\dim_P^n\mu=\dim_P\mu .
\end{equation}
The profiles at integer exponents give the packing dimension of almost
every orthogonal projection, see \eqref{eq:intro_FH}, and Xiao
\cite{xiao1997packing} showed that the profiles at the possibly
non-integer exponent $\alpha m$ give the almost sure packing dimension of
the image of a measure under index-$\alpha$ fractional Brownian motion,
see \eqref{eq:intro_xiao}. We record some elementary properties of the
profiles.

\begin{remark}\label{rem:profiles}
Let $\mu$ be a nonzero finite Borel measure on $\R^n$ and $s>0$.
\begin{enumerate}[(a)]
\item
$\liminf_{r\to0}r^{-t}F_s^\mu(x,r)=0$ for $\mu$-a.e.\ $x$ and every
$t<\dim_P^s\mu$.
\item If $0<s_1\leq s_2$, then
$\min\{1,r^{s_2}|z|^{-s_2}\}\leq\min\{1,r^{s_1}|z|^{-s_1}\}$, so that
$F_{s_2}^\mu\leq F_{s_1}^\mu$ and therefore
$\dim_P^{s_1}\mu\leq\dim_P^{s_2}\mu$.
\item $\dim_P^s\mu\leq\min\{s,\dim_P\mu\}$. Indeed,
$\mu(B(x,r))\leq F_s^\mu(x,r)$, so that \eqref{eq:hu} gives
$\dim_P^s\mu\leq\dim_P\mu$; and for $r\in(0,1)$ we have
$F_s^\mu(x,r)\geq r^s\int\min\{1,|x-y|^{-s}\}\,d\mu(y)$, where the
integral is positive, so that $\dim_P^s\mu\leq s$. Together with (b) and
\eqref{eq:FHcor3}, this gives $\dim_P^s\mu=\dim_P\mu$ for every $s\geq n$.
\item If $0<r\leq u$, then $\ind{B(0,u)}(z)\leq(u/r)^s\min\{1,r^s|z|^{-s}\}$
for every $z$, and therefore
\begin{equation}\label{eq:ball_profile}
\mu\bigl(B(x,u)\bigr)\leq\Bigl(\frac{u}{r}\Bigr)^{s}F_s^\mu(x,r)\qquad\text{for every }x\in\R^n.
\end{equation}
\item \emph{(Layer-cake formula.)} Let $\psi:[0,\infty)\to[0,1]$ be a
continuous nonincreasing function which is constant on $[0,r]$ for some
$r>0$, is piecewise continuously differentiable, and tends to $0$ at
infinity. Then $\psi(v)=\int_v^\infty(-\psi'(u))\,du$ for every $v\geq0$,
and by Tonelli's theorem
\begin{equation}\label{eq:layer_general}
\int\psi(|x-y|)\,d\mu(y)=\int_0^\infty(-\psi'(u))\,\mu\bigl(B(x,u)\bigr)\,du
\qquad\text{for every }x\in\R^n.
\end{equation}
Applied to $\psi(u)=\min\{1,r^su^{-s}\}$ this gives
\begin{equation}\label{eq:layer_profile}
F_s^\mu(x,r)=s\,r^s\int_r^\infty u^{-s-1}\mu\bigl(B(x,u)\bigr)\,du .
\end{equation}
\end{enumerate}
\end{remark}

\noindent
For an analytic set $E\subset\R^n$, let $\mathcal{M}_c^+(E)$ denote the
family of nonzero finite Borel measures with compact support contained
in $E$. It is well known (see \cite{hu1994fractal}, or
\cite[(4.1)]{falconer1997packing}) that
\begin{equation}\label{eq:dimP_sets}
\dim_P E=\sup\{\dim_P\mu : \mu\in\mathcal{M}_c^+(E)\},
\end{equation}
which led Falconer and Howroyd to define the $s$-dimensional packing
dimension profile of $E$ by
$\dim_P^s E:=\sup\{\dim_P^s\mu : \mu\in\mathcal{M}_c^+(E)\}$, see
\cite[(5.3)]{falconer1997packing}. All the dimensions of measures
considered in this paper are invariant under multiplication of the
measure by a positive constant, so the measures in $\mathcal{M}_c^+(E)$
may be normalized to be probability measures.\\

\noindent
Given $0<\alpha<1$ and $m,n\in\mathbb{N}$, an \emph{index-$\alpha$
fractional Brownian motion} is a random field $X:\R^n\to\R^m$ of the form
$X=(X_1,\ldots,X_m)$, where $X_1,\ldots,X_m$ are independent real-valued
random fields $X_i:\R^n\to\R$ satisfying:
\begin{enumerate}
    \item[i)] $X_i(0)=0$;
    \item[ii)] $X_i$ is continuous with probability $1$;
    \item[iii)] the increments $X_i(x)-X_i(y)$ are normally distributed
    with mean $0$ and variance $|x-y|^{2\alpha}$, for all $x,y\in\R^n$.
\end{enumerate}
It is well known that, almost surely, the sample paths of $X$ are
$\beta$-H\"{o}lder continuous on compact sets for every $\beta<\alpha$.
We refer the reader to the classical text of Kahane
\cite{kahane1985some} for a detailed account of fractional Brownian
motion and related stochastic processes.

\section{Estimates for the measure of balls across scales}\label{sec:growth}

In this section we collect the estimates for the measure of balls
$\mu(B(x,r))$ at $\mu$-almost every point $x$ on which most of our
arguments rely. The first one is the elementary bound coming from the
Hausdorff dimension of $\mu$. The second one, Lemma~\ref{lem:growth},
compares $\mu(B(x,r))$ with $\mu(B(x,\rho))$ for $\rho^{a}\leq r\leq 1$.
For the exponent $n(1+\varepsilon)$ it is due to Falconer and Mattila
\cite{falconer1996packing}, and we show that the ambient dimension $n$ can
be replaced by the Assouad dimension of the support. To achieve this, we
first show that every compact set is contained in a compact set
supporting a regular measure whose exponent is arbitrarily close to its
Assouad dimension (Lemma~\ref{lem:regular}); running the argument of
Falconer and Mattila against the resulting regular measure, instead of
the Lebesgue measure, then yields the required growth estimate.

\begin{lemma}[{\cite[Proposition 2.5(a)]{falconer1996packing}}]\label{lem:frostman}
If $0<d<\dim_H\mu$, then for $\mu$-almost every $x$ there is
$c_1=c_1(x)$ such that $\mu(B(x,r))\leq c_1r^d$ for all $r>0$. If $d=0$
this holds with $c_1=1$.
\end{lemma}

The next lemma involves an standard construction, we include the proof by completeness.

\begin{lemma}\label{lem:regular}
Let $E\subset\R^n$ be a compact set and let $\dim_A E<\lambda\leq n$.
Then there exist $t\in(\dim_A E,\lambda)$, a compact set $F\subset\R^n$
with $E\subset F$, a Borel probability measure $\nu$ with $\supp\nu=F$, and
a constant $c\geq1$ such that
\begin{equation}\label{eq:nu_regular}
c^{-1}r^{t}\leq\nu\bigl(B(x,r)\bigr)\quad\text{for } x\in F,\ 0<r\leq1,
\qquad
\nu\bigl(B(x,r)\bigr)\leq c\,r^{t}\quad\text{for } x\in\R^n,\ 0<r\leq1.
\end{equation}
\end{lemma}

\begin{proof}
The statement is invariant under translations and contracting
dilations, up to the value of $c$, so we may assume that
$E\subset[0,1)^n$. Let $s=\dim_A E$, $\varepsilon:=\lambda-s>0$ and
$u:=s+\varepsilon/2<n$. For $j\in\mathbb{N}_0$, let $\mathcal{D}_j$
denote the family of half-open dyadic subcubes of $[0,1)^n$ of side
length $2^{-j}$, and for a bounded set $A$ let $\mathcal{N}_r(A)$ denote
the number of cubes in the mesh of side length $r$, oriented at the
origin, that intersect $A$. By the definition of the Assouad dimension,
there exists $C\geq 1$ such that for all $0<r<R$ and all $x\in E$,
\begin{equation}\label{eq:assouad_bound}
\mathcal{N}_r\bigl(E \cap Q(x,R)\bigr)\leq C \left(\frac{R}{r}\right)^{u},
\end{equation}
where $Q(x,R)$ denotes the cube of side length $R$ centred at $x$.

Fix a positive integer $g$ large enough such that
\begin{equation}\label{def:gap}
    M := \left\lceil C\,2^{u}\,2^{gu} \right\rceil < 2^{gn}
    \qquad\text{and}\qquad
    \frac{u+1+\log_2 C}{g}< \frac{\varepsilon}{2},
\end{equation}
and define $t:=\frac{\log_2 M}{g}$. Since $C\geq 1$, we have
$M\geq 2^{gu}$ and $M\leq C\,2^u\,2^{gu}+1\leq 2C\,2^u\,2^{gu}$, whence
$$s<u\leq t\leq u+\frac{u+1+\log_2 C}{g}<u+\frac{\varepsilon}{2}=\lambda;$$
moreover $t<n$, since $M<2^{gn}$. Set $r_k:=2^{-kg}$ for $k\geq 0$, and
note that
\begin{equation}\label{eq:exact_scaling}
    M^{-k}=r_k^{\,t}\qquad\text{for every }k\geq 0.
\end{equation}

\emph{Construction of $F$.} Let $\mathcal{F}_0:=\{[0,1)^n\}$. Suppose
that for some $k\geq 0$ we have constructed a family of dyadic cubes
$\mathcal{F}_k\subset\mathcal{D}_{kg}$ such that
$E\subset F_k:=\bigcup_{Q\in\mathcal{F}_k}Q$. For each
$Q\in\mathcal{F}_k$, partition $Q$ into its $2^{gn}$ dyadic children in
$\mathcal{D}_{(k+1)g}$, retain all children $Q'$ with
$E\cap Q'\neq\emptyset$, and then add further children of $Q$, chosen
arbitrarily, until exactly $M$ children of $Q$ have been selected. This
is always possible: on the one hand, if $Q\cap E\neq\emptyset$, pick
$x\in Q\cap E$; since $Q\subset Q(x,2r_k)$, inequality
\eqref{eq:assouad_bound} gives
\begin{equation}\label{ineq:children}
    \mathcal{N}_{r_{k+1}}(E\cap Q)
    \leq \mathcal{N}_{r_{k+1}}\bigl(E\cap Q(x,2r_k)\bigr)
    \leq C\,2^u\left(\frac{r_k}{r_{k+1}}\right)^{u}\leq M,
\end{equation}
so all children that intersect $E$ can be retained; on the other hand,
by \eqref{def:gap} we have $M<2^{gn}$, so there are always enough
children available to complete the selection. Let $\mathcal{F}_{k+1}$ be
the family of all selected children of all cubes of $\mathcal{F}_k$, and
$F_{k+1}:=\bigcup_{Q\in\mathcal{F}_{k+1}}Q$. Since every cube of
$\mathcal{D}_{(k+1)g}$ meeting $E$ is a child of some cube of
$\mathcal{F}_k$ (because $E\subset F_k$), we get $E\subset F_{k+1}$, and by
construction $F_{k+1}\subset F_k$ and each cube of $\mathcal{F}_k$
contains exactly $M$ cubes of $\mathcal{F}_{k+1}$; in particular
$\#\mathcal{F}_k=M^k$. Define
\[
F := \bigcap_{k=0}^{\infty} \overline{F_k},
\]
a compact set containing $E$. Since every cube of $\mathcal{F}_k$ has
exactly $M$ selected children in $\mathcal{F}_{k+1}$, setting
\[
\nu(Q):=M^{-k}\qquad \text{for every }Q\in\mathcal{F}_k
\]
defines a consistent mass distribution, which by
\cite[Proposition 1.7]{falconer2013fractal} extends to a Borel probability
measure $\nu$ with $\supp\nu\subset F$. In particular $\nu(F_k)=1$, so
that every cube of $\mathcal{D}_{kg}$ which does not belong to
$\mathcal{F}_k$ is $\nu$-null.

\emph{Upper bound.} Let $x\in\R^n$ and $0<r<1$, and choose $k\geq 0$ such
that $r_{k+1}\leq r<r_k$. Since $B(x,r)\subset Q(x,2r_k)$ and a cube of
side length $2r_k$ meets at most $3^n$ cubes of $\mathcal{D}_{kg}$, each
of which has $\nu$-mass at most $M^{-k}$, we obtain from
\eqref{eq:exact_scaling} and $r_k=2^g r_{k+1}\leq 2^g r$ that
$$\nu(B(x,r))\leq 3^n M^{-k}=3^n r_k^{\,t}\leq 3^n\,2^{gn}\, r^{\,t};$$
for $r=1$ the bound $\nu(B(x,1))\leq 1$ is trivial.

\emph{Lower bound.} Let $x\in F$ and $0<r\leq 1$, and choose $k\geq 0$
with $r_{k+1}\leq \frac{r}{2\sqrt n}<r_k$. Since
$x\in\overline{F_{k+1}}$, which is a finite union of closed cubes, there
exists $Q'\in\mathcal{F}_{k+1}$ with $x\in\overline{Q'}$; as $Q'$ has
diameter $\sqrt n\, r_{k+1}$, we get
$Q'\subset B(x,2\sqrt n\, r_{k+1})\subset B(x,r)$. Hence, using
\eqref{eq:exact_scaling}, $r_{k+1}>2^{-g}\frac{r}{2\sqrt n}$ and $t<n$,
$$\nu(B(x,r))\geq \nu(Q')=r_{k+1}^{\,t}
\geq \left(\frac{2^{-g}}{2\sqrt n}\right)^{n} r^{\,t}.$$
In particular $F\subset\supp\nu$, so that $\supp\nu=F$, and taking
$c:=\max\{3^n 2^{gn},\,(2^{g+1}\sqrt n)^{n}\}$ completes the proof.
\end{proof}

The following lemma sharpens an estimate of Falconer and Mattila
\cite[Corollary 2.3]{falconer1996packing}, replacing the ambient dimension
$n$ in the exponent by the Assouad dimension of the support. The proof follows essentially the same steps as \cite[Corollary 2.3]{falconer1996packing} but using the regular measure provided by Lemma~\ref{lem:regular},
which plays the role of the $n$-dimensional Lebesgue measure in their
argument.

\begin{lemma}\label{lem:growth}
Let $\mu$ be a Borel probability measure on $\R^n$, let $\lambda>N_\mu$
and let $a\in(0,1)$. Then for $\mu$-almost every $x$ there exists
$\rho_0(x)\in(0,1)$ such that
\begin{equation}\label{eq:ball_growth}
\mu\bigl(B(x,r)\bigr)\leq\left(\frac{4r}{\rho}\right)^{\lambda}\mu\bigl(B(x,\rho)\bigr)
\qquad\text{whenever } 0<\rho<\rho_0(x)\ \text{ and }\ \rho^{a}\leq r\leq 1 .
\end{equation}
\end{lemma}

\begin{proof}
Since $4r/\rho\geq 1$ in \eqref{eq:ball_growth}, a larger exponent only
weakens the inequality. If $\lambda>n$, this is
\cite[Corollary 2.3]{falconer1996packing}, stated there with the exponent
$n(1+\varepsilon)$, $\varepsilon=\lambda/n-1$; it holds for every Borel
probability measure on $\R^n$. Hence we may assume that
$N_\mu<\lambda\leq n$, so that $E:=\supp\mu$ is compact and
$\dim_A E<\lambda$. Let $t\in(\dim_A E,\lambda)$, $F\supset E$, $\nu$ and
$c\geq1$ be given by Lemma~\ref{lem:regular}, and fix $\eta>0$ with
$t(1+\eta)\leq\lambda$. It suffices to prove \eqref{eq:ball_growth} with
the exponent $t(1+\eta)$ in place of $\lambda$.

For nonnegative integers $p<q$ define
$$A_{p,q}:=\bigl\{x\in E:\ \mu(B(x,2^{-p}))>2^{(q-p)t(1+\eta)}\,\mu(B(x,2^{-q}))\bigr\}.$$
We claim that
\begin{equation}\label{eq:dyadic_bound}
\mu(A_{p,q})\leq 4^{t}c^{2}\,2^{-(q-p)t\eta}.
\end{equation}
First observe that for every $x\in\R^n$,
\begin{equation}\label{eq:local_bound}
\mu\bigl(A_{p,q}\cap B(x,2^{-q-1})\bigr)\leq
2^{(p-q)t(1+\eta)}\,\mu\bigl(B(x,2^{-p+1})\bigr).
\end{equation}
Indeed, if the left-hand side is zero there is nothing to prove;
otherwise pick $y\in A_{p,q}\cap B(x,2^{-q-1})$. Any two points of
$B(x,2^{-q-1})$ lie within distance $2^{-q}$ of each other, so
$A_{p,q}\cap B(x,2^{-q-1})\subset B(y,2^{-q})$, and since $y\in A_{p,q}$
means $\mu(B(y,2^{-q}))<2^{(p-q)t(1+\eta)}\mu(B(y,2^{-p}))$, we get
\begin{align*}
    \mu\bigl(A_{p,q}\cap B(x,2^{-q-1})\bigr)&\leq\mu(B(y,2^{-q}))
    <2^{(p-q)t(1+\eta)}\mu(B(y,2^{-p}))\\
    &\leq 2^{(p-q)t(1+\eta)}\mu(B(x,2^{-p+1})).
\end{align*}
Now we estimate $\mu(A_{p,q})$. By the lower bound in
\eqref{eq:nu_regular}, valid at every point of $E\subset F$,
$$\mu(A_{p,q})\leq c\,2^{(q+1)t}\int_{A_{p,q}}\nu\bigl(B(x,2^{-q-1})\bigr)\,d\mu(x).$$
Writing $\nu(B(x,2^{-q-1}))=\int\ind{\{|x-y|\leq 2^{-q-1}\}}\,d\nu(y)$
and applying Fubini's theorem and \eqref{eq:local_bound},
\begin{align*}
\int_{A_{p,q}}\nu\bigl(B(x,2^{-q-1})\bigr)\,d\mu(x)
&=\int \mu\bigl(A_{p,q}\cap B(y,2^{-q-1})\bigr)\,d\nu(y)\\
&\leq 2^{(p-q)t(1+\eta)}\int\mu\bigl(B(y,2^{-p+1})\bigr)\,d\nu(y).
\end{align*}
Applying Fubini's theorem once more and then the upper bound in
\eqref{eq:nu_regular} (for $p=0$ simply use $\nu(B(x,2))\leq1\leq c\,2^t$),
$$\int\mu\bigl(B(y,2^{-p+1})\bigr)\,d\nu(y)
=\int\nu\bigl(B(x,2^{-p+1})\bigr)\,d\mu(x)
\leq c\,2^{t}\,2^{-pt}.$$
Combining,
$$\mu(A_{p,q})\leq c\,2^{(q+1)t}\, 2^{(p-q)t(1+\eta)}\, c\,2^{t}\,2^{-pt}
= 4^{t}c^{2}\,2^{-(q-p)t\eta},$$
which proves \eqref{eq:dyadic_bound}.

Now, for $q_0\in\mathbb{N}$ let
$$
M_{q_0}:=\mu\Bigl(\bigcup\bigl\{A_{p,q}\,:\,p,q\in\mathbb{N}_0,\ 0\leq p\leq aq,\ q\geq q_0\bigr\}\Bigr).
$$
By \eqref{eq:dyadic_bound}, using $q-p\geq q-aq=(1-a)q$ and summing the
geometric series in $p$,
$$M_{q_0}\leq 4^{t}c^{2}\sum_{q=q_0}^{\infty}\sum_{p=0}^{\lfloor aq\rfloor}
2^{-(q-p)t\eta}
\leq 4^{t}c^{2}\,\frac{2^{t\eta}}{2^{t\eta}-1}\sum_{q=q_0}^{\infty}2^{-q(1-a)t\eta}
=c_{1}\,2^{-q_{0}(1-a)t\eta},$$
where $c_{1}$ depends only on $a,\eta,t$ and $c$.

Finally, let $\rho_{0}:=2^{-q_{0}-1}$, and suppose that $x\in E$,
$0<\rho<\rho_{0}$, $\rho^{a}\leq r\leq 1$ and
$$\mu(B(x,r))>\left(\frac{4r}{\rho}\right)^{t(1+\eta)}\mu(B(x,\rho)).$$
Choose integers $p,q$ with $2^{-p-1}<r\leq 2^{-p}$ and
$2^{-q}\leq\rho<2^{-q+1}$. Then $p\geq0$, $2^{-p}\geq r\geq\rho^{a}\geq
2^{-aq}$ gives $p\leq aq$, and $2^{-q}\leq\rho<2^{-q_{0}-1}$ gives
$q>q_{0}$. Moreover $\frac{4r}{\rho}\geq\frac{4\cdot 2^{-p-1}}{2^{-q+1}}
=2^{q-p}$, so
$$
2^{(q-p)t(1+\eta)}\mu(B(x,2^{-q}))\leq \left(\tfrac{4r}{\rho}\right)^{t(1+\eta)}\mu(B(x,\rho))
<\mu(B(x,r))\leq \mu(B(x,2^{-p})),
$$
that is, $x\in A_{p,q}$. Hence the set of points $x\in E$ for which
\eqref{eq:ball_growth} fails, with $t(1+\eta)$ in place of $\lambda$, for
some $0<\rho<\rho_{0}$ and $\rho^{a}\leq r\leq 1$, has $\mu$-measure at
most $M_{q_{0}}\leq c_{1}2^{-q_{0}(1-a)t\eta}$. These exceptional sets
decrease as $q_{0}\to\infty$ and their measures tend to $0$; hence their
intersection is $\mu$-null, which means that for $\mu$-a.e.\ $x$ there
exists $\rho_{0}(x)\in(0,1)$ such that
$\mu(B(x,r))\leq(4r/\rho)^{t(1+\eta)}\mu(B(x,\rho))$ for all
$0<\rho<\rho_{0}(x)$ and $\rho^{a}\leq r\leq 1$.
\end{proof}

\begin{remark}\label{rem:uniform}
The estimate \eqref{eq:ball_growth} holds for \emph{every}
$\rho\in(0,\rho_0(x))$, and not only along a sequence of scales depending
on $x$. This uniformity will be used repeatedly.
\end{remark}

\section{Packing intermediate dimensions}\label{sec:pid}

In this section we introduce the packing intermediate dimensions and the
quasi-Hausdorff dimension of a measure, and we establish their basic
properties.

\begin{definition}\label{def:pid}
Let $\theta\in(0,1]$, $0\leq t\leq n$ and $r\in(0,1)$. For $z\in\R^n$ let
\[
\phi^{t,n}_{r,\theta}(z)=\min\bigl\{1,\ r^t|z|^{-t},\ r^{\theta(n-t)+t}|z|^{-n}\bigr\}
=\begin{cases}
1, & |z|\leq r,\\
r^t|z|^{-t}, & r<|z|\leq r^\theta,\\
r^{\theta(n-t)+t}|z|^{-n}, & |z|>r^\theta,
\end{cases}
\]
and write $\phi^{t,n}_{r,\theta}(u)$, $u\geq0$, for its value at any $z$
with $|z|=u$. For a nonzero finite Borel measure $\nu$ on $\R^n$ let
$$F^\nu_{t,n,\theta}(x,r)=\int\phi^{t,n}_{r,\theta}(x-y)\,d\nu(y),$$
and define the \emph{$\theta$-packing intermediate dimension} of $\nu$ by
\[
\dim_{P,\theta}\nu=\sup\Big\{t\in[0,n] : \liminf_{r\to0}r^{-t}F^\nu_{t,n,\theta}(x,r)=0\ \text{ for $\nu$-a.e.\ } x\in\R^n\Big\}.
\]
Since $0<\phi^{t,n}_{r,\theta}\leq1$, the definition makes sense for every
finite Borel measure, and its value does not change if $\nu$ is
multiplied by a positive constant.
\end{definition}

\begin{remark}\label{rem:basic}
Let $\nu$ be a nonzero finite Borel measure on $\R^n$,
$\theta\in(0,1]$, $0\leq t\leq n$ and $r\in(0,1)$.
\begin{enumerate}[(a)]
\item If $0\leq t'\leq t$, then
\[
r^{-t'}\phi^{t',n}_{r,\theta}\leq r^{\theta(t-t')}\,r^{-t}\phi^{t,n}_{r,\theta}.
\]
Consequently, $\liminf_{r\to0}r^{-t}F^\nu_{t,n,\theta}(x,r)=0$ for
$\nu$-a.e.\ $x$ and every $t<\dim_{P,\theta}\nu$.
\item If $\theta\leq\theta'\leq1$, then
$\phi^{t,n}_{r,\theta'}\leq\phi^{t,n}_{r,\theta}$. Hence
$\theta\mapsto\dim_{P,\theta}\nu$ is nondecreasing on $(0,1]$.
\item $\phi^{t,n}_{r,1}(z)=\min\{1,r^n|z|^{-n}\}$, so that
$F^\nu_{t,n,1}=F^\nu_n$ and $\dim_{P,1}\nu=\dim_P^n\nu=\dim_P\nu$ by
\eqref{eq:FHcor3}.
\item If $u\in[r,r^\theta]$, then
$\ind{B(0,u)}\leq (u/r)^t\,\phi^{t,n}_{r,\theta}$, and therefore
\begin{equation}\label{eq:ballkernel}
\nu(B(x,u))\leq \Big(\frac ur\Big)^t F^\nu_{t,n,\theta}(x,r)\qquad\text{for every } x\in\R^n.
\end{equation}
\item For every $x\in\R^n$,
\begin{equation}\label{eq:layer}
F^\nu_{t,n,\theta}(x,r)=t\,r^t\int_r^{r^\theta}u^{-t-1}\nu(B(x,u))\,du+n\,r^{\theta(n-t)+t}\int_{r^\theta}^\infty u^{-n-1}\nu(B(x,u))\,du .
\end{equation}
\end{enumerate}

\end{remark}

We now give an equivalent characterization of the packing intermediate
dimensions that dispenses with the kernels, $\dim_{P,\theta}\nu$ is the
critical exponent for Frostman-type estimates of $\nu(B(x,r))$ along
windows of scales of the form $[\rho,\rho^{\theta}]$.

\begin{theorem}\label{thm:windows}
Let $\nu$ be a nonzero finite Borel measure on $\R^n$ and let
$\theta\in(0,1]$.
\begin{enumerate}[(i)]
\item If $0\leq t<\dim_{P,\theta}\nu$, then for $\nu$-almost every $x$
and every $\varepsilon>0$ there are arbitrarily small $\rho\in(0,1)$ such
that
\[
\nu(B(x,r))\leq \varepsilon\,r^t\qquad\text{for all } r\in[\rho,\rho^\theta].
\]
\item $\dim_{P,\theta}\nu$ is the supremum of the numbers $t\geq0$ with
the following property: for $\nu$-almost every $x$, for every
$\varepsilon>0$ and every $r_0>0$ there is $\rho\in(0,r_0)$ such that
$\nu(B(x,r))\leq\varepsilon r^t$ for all $r\in[\rho,\rho^\theta]$.
\item $\dim_{P,\theta}\nu$ is also the supremum of the numbers $t\geq0$
such that, for $\nu$-almost every $x$, there are arbitrarily small
$\rho\in(0,1)$ with $\nu(B(x,r))\leq r^t$ for all $r\in[\rho,\rho^\theta]$.
\end{enumerate}
\end{theorem}

\begin{proof}
Statements (i) and (ii) do not change if $\nu$ is multiplied by a
positive constant, so in their proof we may assume that $\nu$ is a
probability measure.

(i) By Remark~\ref{rem:basic}(a), for $\nu$-almost every $x$ we have
$\liminf_{r\to0}r^{-t}F^\nu_{t,n,\theta}(x,r)=0$. For such $x$ and
$\varepsilon>0$ there are arbitrarily small $\rho\in(0,1)$ with
$F^\nu_{t,n,\theta}(x,\rho)\leq\varepsilon\rho^t$, and by
\eqref{eq:ballkernel}, $\nu(B(x,r))\leq(r/\rho)^t\varepsilon\rho^t=\varepsilon r^t$
for every $r\in[\rho,\rho^\theta]$.

(ii) Let $t^*$ denote the supremum in (ii). By (i),
$\dim_{P,\theta}\nu\leq t^*$. Conversely, let $t>0$ have the property in
(ii). Then $\liminf_{r\to0}r^{-t}\nu(B(x,r))=0$ for $\nu$-a.e.\ $x$, so
that $t\leq\dim_P\nu\leq n$ by \eqref{eq:hu} and \eqref{eq:chain}. Let
$0\leq t'<t$; we show that $t'\leq\dim_{P,\theta}\nu$. Fix
$\lambda\in(n,n+t-t')$ and $a\in(0,1)$ with $(1-a)\lambda<t-t'$, and let
$\rho_0(x)$ be given by Lemma~\ref{lem:growth} for these parameters
(note that $\lambda>n\geq N_\nu$). Fix $x$ in the intersection of the two
sets of full measure, let $\varepsilon\in(0,1)$, and choose $\rho\in(0,1)$
with $\rho^\theta<\rho_0(x)$ and $\nu(B(x,r))\leq\varepsilon r^t$ for all
$r\in[\rho,\rho^\theta]$. By \eqref{eq:layer}, with $t'$ in place of $t$,
we have $F^\nu_{t',n,\theta}(x,\rho)=J_1+J_2+J_3+J_4$, where $J_1$ is the
first term in \eqref{eq:layer} and $J_2,J_3,J_4$ are the parts of the
second term corresponding to the intervals $[\rho^\theta,\rho^{a\theta}]$,
$[\rho^{a\theta},1]$ and $[1,\infty)$. The window condition gives
\[
J_1\leq t'\varepsilon\rho^{t'}\int_\rho^{\rho^\theta}u^{t-t'-1}\,du\leq\frac{t'}{t-t'}\,\varepsilon\rho^{t'} .
\]
For $u\in[\rho^\theta,\rho^{a\theta}]$, applying \eqref{eq:ball_growth} at
the scale $\rho^\theta$ and then the window condition at $r=\rho^\theta$,
$$\nu(B(x,u))\leq\nu(B(x,\rho^{a\theta}))\leq 4^\lambda\rho^{-(1-a)\theta\lambda}\,\nu(B(x,\rho^\theta))\leq4^\lambda\varepsilon\,\rho^{\theta t-(1-a)\theta\lambda},$$
so that, since $\int_{\rho^\theta}^\infty u^{-n-1}du=\rho^{-\theta n}/n$,
\[
J_2\leq 4^\lambda\varepsilon\,\rho^{t'+\theta(t-t')-(1-a)\theta\lambda}\leq4^\lambda\varepsilon\rho^{t'} .
\]
For $u\in[\rho^{a\theta},1]$, in the same way,
$\nu(B(x,u))\leq(4u/\rho^\theta)^\lambda\nu(B(x,\rho^\theta))\leq4^\lambda\varepsilon u^\lambda\rho^{\theta(t-\lambda)}$,
and since $\lambda>n$,
\[
J_3\leq n\,\rho^{\theta(n-t')+t'}\,4^\lambda\varepsilon\,\rho^{\theta(t-\lambda)}\int_0^1u^{\lambda-n-1}\,du
=\frac{n4^\lambda}{\lambda-n}\,\varepsilon\,\rho^{t'+\theta(t-t')-\theta(\lambda-n)}\leq\frac{n4^\lambda}{\lambda-n}\,\varepsilon\rho^{t'} .
\]
Finally, $\nu(B(x,u))\leq1$ gives $J_4\leq\rho^{t'+\theta(n-t')}$.
Therefore
\[
F^\nu_{t',n,\theta}(x,\rho)\leq C\varepsilon\rho^{t'}+\rho^{t'+\theta(n-t')},\qquad
C=\frac{t'}{t-t'}+4^\lambda+\frac{n4^\lambda}{\lambda-n},
\]
where $C$ depends neither on $\varepsilon$ nor on $\rho$. Since $t'<n$
and $\rho$ can be taken arbitrarily small, we get
$\liminf_{r\to0}r^{-t'}F^\nu_{t',n,\theta}(x,r)\leq C\varepsilon$, and
since $\varepsilon$ was arbitrary, this $\liminf$ vanishes. This holds
for $\nu$-almost every $x$, so $t'\leq\dim_{P,\theta}\nu$. Letting
$t'\to t$ gives $t\leq\dim_{P,\theta}\nu$, and therefore
$t^*\leq\dim_{P,\theta}\nu$.

(iii) Let $t^\circ$ denote the supremum in (iii). By (i) with
$\varepsilon=1$, $\dim_{P,\theta}\nu\leq t^\circ$. Conversely, let $t$
have the property in (iii) and $t'<t$. For $r\in[\rho,\rho^\theta]$ we have
$\nu(B(x,r))\leq r^{t-t'}r^{t'}\leq\rho^{\theta(t-t')}r^{t'}$, and
$\rho^{\theta(t-t')}\to0$ as $\rho\to0$, so $t'$ has the property in
(ii). Hence $t'\leq\dim_{P,\theta}\nu$, and $t^\circ\leq\dim_{P,\theta}\nu$.
\end{proof}

\begin{remark}\label{rem:windows}
\begin{enumerate}[(a)]
\item By Theorem~\ref{thm:windows}(ii), the value of
$\dim_{P,\theta}\nu$ depends only on the function
$(x,r)\mapsto\nu(B(x,r))$ on $\supp\nu$. In particular, it does not depend on the dimension of the ambient space: if
$\R^n$ is regarded as a subspace of $\R^{n'}$, $n'\geq n$, the packing
intermediate dimensions of $\nu$ computed in $\R^{n'}$, with the kernels
$\phi^{t,n'}_{r,\theta}$, coincide with those computed in $\R^n$. For
$\theta=1$ the window $[\rho,\rho^\theta]$ reduces to the single scale
$\rho$, and Theorem~\ref{thm:windows}(ii) reduces to \eqref{eq:hu}, in
accordance with Remark~\ref{rem:basic}(c).
\item The condition in Theorem~\ref{thm:windows}(ii), written with
windows $[\delta^{1/\theta},\delta]$, $\delta=\rho^\theta$, involves the
same two-scale Frostman-type condition used in
\cite[Definition 3.1]{ANGELINI2026130039} to define the Minkowski
intermediate dimensions of a measure. The difference lies in the
quantifiers. There, the constant and the scales are required to be
uniform over the support (see also \cite[Lemma 3.3]{ANGELINI2026130039}),
whereas here the scales are allowed to depend on the point $x$. We also
note that the kernels $\phi^{t,n}_{r,\theta}$ depend on the ambient
dimension $n$ rather than on the Assouad dimension of $\nu$, thus no regularity hypothesis on $\nu$
is required, whereas the capacity-theoretic characterization of the
intermediate dimensions of measures in \cite{ANGELINI2026130039} requires
$\dim_A\nu<\infty$, that is, $\nu$ doubling (see
\cite[Lemma 4.1.1]{fraser2020assouad}).
\end{enumerate}
\end{remark}

As a first consequence of the kernel formulation, the packing
intermediate dimensions do not increase under Lipschitz maps.

\begin{proposition}\label{prop:lipschitz}
Let $\nu$ be a nonzero finite Borel measure on $\R^n$, let
$\theta\in(0,1]$, and let $f:\supp\nu\to\R^m$ be a Lipschitz map. Then
$\dim_{P,\theta}\nu_f\leq\dim_{P,\theta}\nu$. In particular,
$\dim_{P,\theta}\nu_V\leq\dim_{P,\theta}\nu$ for every $V\in G(n,m)$.
\end{proposition}

\begin{proof}
By Remark~\ref{rem:windows}(a), we may compute both dimensions in $\R^k$,
$k=\max\{n,m\}$, regarding $\R^n$ and $\R^m$ as subspaces of $\R^k$. Let
$L\geq1$ be such that $|f(x)-f(y)|\leq L|x-y|$ for $x,y\in\supp\nu$.
Since $\phi^{t,k}_{r,\theta}$ is radially nonincreasing, the expression
of the kernel as a minimum gives, for $x,y\in\supp\nu$ and $0\leq t\leq k$,
$$\phi^{t,k}_{r,\theta}\bigl(f(x)-f(y)\bigr)\geq\phi^{t,k}_{r,\theta}\bigl(L(x-y)\bigr)\geq L^{-k}\,\phi^{t,k}_{r,\theta}(x-y),$$
and integrating with respect to $\nu(y)$,
$F^{\nu_f}_{t,k,\theta}(f(x),r)\geq L^{-k}F^{\nu}_{t,k,\theta}(x,r)$ for
every $x\in\supp\nu$. Now let $\dim_{P,\theta}\nu<t\leq k$ (if
$\dim_{P,\theta}\nu=k$ there is nothing to prove). By definition, the set
$A$ of points $x\in\supp\nu$ with
$\liminf_{r\to0}r^{-t}F^\nu_{t,k,\theta}(x,r)>0$ has positive
$\nu$-measure, and
$A\subset f^{-1}\bigl(\{w:\liminf_{r\to0}r^{-t}F^{\nu_f}_{t,k,\theta}(w,r)>0\}\bigr)$.
Hence the latter set has positive $\nu_f$-measure, so that $t$ does not
belong to the set defining $\dim_{P,\theta}\nu_f$, and by
Remark~\ref{rem:basic}(a), $\dim_{P,\theta}\nu_f\leq t$.
\end{proof}

the proof of the next lemma is straightforward.

\begin{lemma}\label{lem:compare}
Let $\nu$ be a nonzero finite Borel measure on $\R^n$. For every
$\theta\in(0,1]$ we have $\dim_H\nu\leq\dim_{P,\theta}\nu\leq\dim_P\nu$.
\end{lemma}

The following universal lower bound is a key tool for estimating the
packing intermediate dimensions. Recall the definition of $N_\mu$ in
\eqref{eq:Nmu}.

\begin{lemma}\label{lem:floor}
Let $\mu$ be a Borel probability measure on $\R^n$, and write
$P=\dim_P\mu$ and $N=N_\mu$. Then, for every $\theta\in(0,1]$,
\begin{equation}\label{eq:floor}
\dim_{P,\theta}\mu\ \geq\ \sigma_\theta(P):=\frac{\theta\,N\,P}{N-(1-\theta)P},
\end{equation}
where $\sigma_\theta(P):=0$ if $P=0$. Note that if $P>0$, then
$N-(1-\theta)P\geq\theta N>0$.
\end{lemma}

\begin{proof}
If $P=0$ there is nothing to prove, so assume $P>0$. Let $0<p<P$,
$\lambda>N$ and
\[
0<t<\frac{\theta\lambda p}{\lambda-(1-\theta)p}.
\]
Since $\theta\lambda\leq\lambda-(1-\theta)p$, we have $t<p$. A direct
computation shows that the condition on $t$ is equivalent to
$t(\lambda-p)<\theta p(\lambda-t)$, that is,
\[
\frac{\lambda-p}{\lambda-t}<\frac{\theta p}{t}.
\]
Fix $\kappa$ strictly between these two numbers and fix $a\in(0,1)$ with
$(1-a)\lambda<p-t$. Then the three exponents
\[
e_1=p-\frac{\kappa t}{\theta},\qquad e_2=p-(1-a)\lambda-t,\qquad e_3=\kappa(\lambda-t)-(\lambda-p)
\]
are positive. Since $p<\dim_P\mu$, by \eqref{eq:hu} for $\mu$-almost
every $x$ there are arbitrarily small $\delta$ with
$\mu(B(x,\delta))\leq\delta^p$, and Lemma~\ref{lem:growth} applies with
the parameters $\lambda$ and $a$. Fix such an $x$, let $\varepsilon>0$
and $r_0>0$, and choose $\delta\in(0,\rho_0(x))$, where $\rho_0(x)$ is
given by Lemma~\ref{lem:growth}, with $\mu(B(x,\delta))\leq\delta^p$,
$\delta^{\kappa/\theta}<r_0$ and
$4^\lambda\max\{\delta^{e_1},\delta^{e_2},\delta^{e_3}\}\leq\varepsilon$.
Put $\rho=\delta^{\kappa/\theta}$, so that the window is
$[\rho,\rho^\theta]=[\delta^{\kappa/\theta},\delta^\kappa]$, and let
$r\in[\rho,\rho^\theta]$.

If $r\leq\delta$ then, since $r\geq\delta^{\kappa/\theta}$,
\[
\mu(B(x,r))\leq\mu(B(x,\delta))\leq\delta^p=\delta^{e_1}\delta^{\kappa t/\theta}\leq\delta^{e_1}\,r^t .
\]
If $\delta<r\leq\delta^a$ then, by \eqref{eq:ball_growth} with radius
$\delta^a$,
\[
\mu(B(x,r))\leq\mu(B(x,\delta^a))\leq4^\lambda\delta^{(a-1)\lambda}\delta^p=4^\lambda\delta^{e_2}\delta^{t}\leq4^\lambda\delta^{e_2}\,r^t .
\]
If $\delta^a<r\leq\delta^\kappa$ then $r\leq1$ and, by
\eqref{eq:ball_growth} and $r\leq\delta^\kappa$,
\[
\mu(B(x,r))\leq4^\lambda\Big(\frac r\delta\Big)^\lambda\delta^p=4^\lambda r^t\,r^{\lambda-t}\delta^{p-\lambda}\leq4^\lambda r^t\delta^{\kappa(\lambda-t)+p-\lambda}=4^\lambda\delta^{e_3}\,r^t .
\]
Hence $\mu(B(x,r))\leq\varepsilon r^t$ for all $r\in[\rho,\rho^\theta]$,
and $\rho<r_0$. By Theorem~\ref{thm:windows}(ii), $t\leq\dim_{P,\theta}\mu$.
Letting $t\to\theta\lambda p/(\lambda-(1-\theta)p)$, then $p\to P$ and
$\lambda\to N$ gives the claim; the expression is continuous there
because $N-(1-\theta)P\geq\theta N>0$.
\end{proof}

In particular, if $\dim_P\mu=N_\mu$ then $\sigma_\theta(P)=P$, and
therefore $\dim_{P,\theta}\mu=\dim_P\mu$ for every $\theta\in(0,1]$. Note
also that $\sigma_\theta(P)$ is nonincreasing as a function of $N$.

Since $\theta\mapsto\dim_{P,\theta}\nu$ is nondecreasing, the following
limit exists.

\begin{definition}\label{def:qH}
The \emph{quasi-Hausdorff dimension} of a nonzero finite Borel measure
$\nu$ on $\R^n$ is
$$\dim_{qH}\nu:=\lim_{\theta\to0}\dim_{P,\theta}\nu=\inf_{\theta\in(0,1]}\dim_{P,\theta}\nu .$$
\end{definition}

\begin{remark}\label{rem:qH}
\begin{enumerate}[(a)]
\item By Lemma~\ref{lem:compare},
$\dim_H\nu\leq\dim_{qH}\nu\leq\dim_P\nu$. Thus the packing intermediate
dimensions range from the packing dimension, attained at $\theta=1$, to
the quasi-Hausdorff dimension, which is at least the Hausdorff
dimension. The name is chosen by analogy with the quasi-Assouad
dimension, which is defined as a limit of the Assouad spectrum
\cite{fraser2018new}, and with the quasi-Hausdorff dimensions of sets
introduced in \cite{angelini2025critical}, which are the limits as
$\theta\to0$ of the intermediate dimensions of
\cite{falconer2020intermediate}; the latter coincide with the Hausdorff
dimension at $\theta=0$, but may be discontinuous there.
\item By Lemma~\ref{lem:floor}, $\dim_{qH}\mu=\dim_P\mu$ whenever
$\dim_P\mu=N_\mu$; for instance, if $\dim_P\mu=n$, or if $\supp\mu$ is
compact and $\dim_P\mu=\dim_A(\supp\mu)$. Thus the first inequality in
(a) is strict for every measure with $\dim_H\mu<\dim_P\mu=n$. The second inequality in (a) can also be strict, see
Examples \ref{ex:constant} and ~\ref{ex:FM} below.
\item By Proposition~\ref{prop:lipschitz}, $\dim_{qH}\nu_f\leq\dim_{qH}\nu$
for every Lipschitz map $f:\supp\nu\to\R^m$.
\end{enumerate}
\end{remark}

\section{Applications to packing dimension profiles}
\subsection{Packing dimension profiles and the quasi-Hausdorff dimension}\label{sec:threshold}

We now prove our first main result: the quasi-Hausdorff dimension is the
exact threshold at which the packing dimension profiles attain their
maximal value. The proof combines two elementary comparisons. On the one
hand, the kernel of the $s$-dimensional profile is dominated by the
kernel $\phi^{t,n}_{r,\theta}$ up to an error $r^{t(1-\theta)}$; on the
other hand, a small value of the profile at a scale $r$ forces
Frostman-type estimates along the whole window $[r,r^\theta]$, which by
Theorem~\ref{thm:windows} give lower bounds for $\dim_{P,\theta}\mu$.

\begin{theorem}\label{thm:threshold}
Let $\mu$ be a Borel probability measure on $\R^n$ and let $0<s\leq n$.
Then, for every $\theta\in(0,1]$,
\begin{equation}\label{eq:two_sided}
(1-\theta)\min\{s,\dim_{P,\theta}\mu\}\ \leq\ \dim_P^{s}\mu\ \leq\ (1-\theta)\,s+\theta\min\{s,\dim_{P,\theta}\mu\}.
\end{equation}
Consequently,
\begin{equation}\label{eq:qH_bounds}
\min\{s,\dim_{qH}\mu\}\ \leq\ \dim_P^{s}\mu\ \leq\ \min\{s,\dim_P\mu\},
\end{equation}
and
\begin{equation}\label{eq:threshold}
\dim_P^{s}\mu=s\quad\Longleftrightarrow\quad s\leq\dim_{qH}\mu .
\end{equation}
\end{theorem}

\begin{proof}
We first prove the lower bound in \eqref{eq:two_sided}. For
$0\leq t\leq s$, $r\in(0,1)$ and $z\in\R^n$ we have
\begin{equation}\label{eq:kernel_comparison}
\min\{1,r^s|z|^{-s}\}\leq\phi^{t,n}_{r,\theta}(z)+r^{t(1-\theta)}.
\end{equation}
Indeed, if $|z|\leq r^\theta$, then
$\min\{1,r^s|z|^{-s}\}\leq\min\{1,r^t|z|^{-t}\}=\phi^{t,n}_{r,\theta}(z)$,
and if $|z|>r^\theta$, then
$\min\{1,r^s|z|^{-s}\}\leq r^{s(1-\theta)}\leq r^{t(1-\theta)}$.
Integrating \eqref{eq:kernel_comparison} with respect to $\mu$,
\begin{equation}\label{eq:comparison_profiles}
F_s^\mu(x,r)\leq F^\mu_{t,n,\theta}(x,r)+r^{t(1-\theta)}\qquad\text{for all }x\in\R^n,\ r\in(0,1).
\end{equation}
Let $0<t<\min\{s,\dim_{P,\theta}\mu\}$. By Remark~\ref{rem:basic}(a), for
$\mu$-a.e.\ $x$ there exists a sequence $r_k\to0$ with
$F^\mu_{t,n,\theta}(x,r_k)\leq r_k^t$, and \eqref{eq:comparison_profiles}
gives $F_s^\mu(x,r_k)\leq2\,r_k^{t(1-\theta)}$. Now observe that $F_s^\mu$
does not depend on $t$; therefore, for every $t'<t(1-\theta)$,
$$r_k^{-t'}F_s^{\mu}(x,r_k)\leq 2\,r_k^{\,t(1-\theta)-t'}
\longrightarrow 0\qquad (k\to\infty),$$
so that $\liminf_{r\to0}r^{-t'}F_s^\mu(x,r)=0$ for $\mu$-a.e.\ $x$, and
$t'\leq\dim_P^s\mu$. Letting $t'\to t(1-\theta)$ and then
$t\to\min\{s,\dim_{P,\theta}\mu\}$ gives the lower bound.

For the upper bound in \eqref{eq:two_sided}, since $\dim_P^s\mu\leq s$ by
Remark~\ref{rem:profiles}(c), it suffices to show that
\begin{equation}\label{eq:pid_from_profile}
\dim_{P,\theta}\mu\ \geq\ s-\frac{s-\dim_P^s\mu}{\theta},
\end{equation}
which is equivalent to $\dim_P^s\mu\leq(1-\theta)s+\theta\dim_{P,\theta}\mu$.
Let $0\leq t<\dim_P^s\mu$ and set $t':=s-(s-t)/\theta\leq t$; we may
assume that $t'>0$. By Remark~\ref{rem:profiles}(a), for $\mu$-a.e.\ $x$
and every $\varepsilon>0$ there are arbitrarily small $r\in(0,1)$ with
$F_s^\mu(x,r)\leq\varepsilon r^t$. For such $r$ and every $u\in[r,r^\theta]$,
\eqref{eq:ball_profile} gives
$$\mu(B(x,u))\leq\Bigl(\frac ur\Bigr)^{s}F_s^\mu(x,r)\leq\varepsilon\,u^{s}r^{t-s}
=\varepsilon\,u^{t'}\,u^{s-t'}r^{t-s}\leq\varepsilon\,u^{t'}r^{\theta(s-t')+t-s}=\varepsilon\,u^{t'},$$
where we used $u\leq r^\theta$, $s-t'\geq0$ and $\theta(s-t')=s-t$. Thus
$t'$ has the property in Theorem~\ref{thm:windows}(ii), and therefore
$t'\leq\dim_{P,\theta}\mu$. Letting $t\to\dim_P^s\mu$ gives
\eqref{eq:pid_from_profile}.

As $\theta\downarrow0$, $\dim_{P,\theta}\mu$ decreases to
$\dim_{qH}\mu$, so the lower bound in \eqref{eq:two_sided} gives
$\dim_P^s\mu\geq\min\{s,\dim_{qH}\mu\}$; the upper bound in
\eqref{eq:qH_bounds} is Remark~\ref{rem:profiles}(c). Finally, if
$s\leq\dim_{qH}\mu$, then \eqref{eq:qH_bounds} gives $\dim_P^s\mu=s$;
and if $s>\dim_{qH}\mu$, we may choose $\theta\in(0,1]$ with
$\dim_{P,\theta}\mu<s$, and then \eqref{eq:two_sided} gives
$\dim_P^s\mu\leq(1-\theta)s+\theta\dim_{P,\theta}\mu<s$.
\end{proof}

\begin{remark}\label{rem:threshold}
By \eqref{eq:threshold}, the quasi-Hausdorff dimension is determined by
the profiles:
$$\dim_{qH}\mu=\sup\{s\in(0,n]\,:\,\dim_P^s\mu=s\},$$
with the convention $\sup\emptyset=0$. In view of
\eqref{eq:intro_xiao}, it can therefore be read off from the almost sure
packing dimensions of the images of $\mu$ under fractional Brownian
motions, see Corollary~\ref{cor:images_threshold}.
\end{remark}

In the extremal case $\dim_{qH}\mu=\dim_P\mu$ the profiles admit a closed
form, which is the packing analogue, at the level of profiles, of the
formulas \eqref{eq:intro_HT} and \eqref{eq:kahane}; moreover, this is the
only case in which such a closed form holds.

\begin{corollary}\label{cor:kahane_profiles}
Let $\mu$ be a Borel probability measure on $\R^n$. The following are
equivalent:
\begin{enumerate}[(i)]
\item $\dim_{qH}\mu=\dim_P\mu$;
\item $\dim_P^s\mu=\min\{s,\dim_P\mu\}$ for every $s>0$;
\item $\dim_P\mu=0$, or $\dim_P^{s}\mu=s$ for $s=\dim_P\mu$.
\end{enumerate}
In particular, these conditions hold whenever $\dim_P\mu=N_\mu$, for
instance if $\supp\mu$ is compact and $\dim_P\mu=\dim_A(\supp\mu)$.
\end{corollary}

\begin{proof}
If (i) holds, then \eqref{eq:qH_bounds} gives (ii) for $0<s\leq n$, and
for $s\geq n$ we have $\dim_P^s\mu=\dim_P\mu$ by
Remark~\ref{rem:profiles}(c). Clearly (ii) implies (iii). If (iii) holds
and $\dim_P\mu>0$, then \eqref{eq:threshold} with $s=\dim_P\mu\in(0,n]$
gives $\dim_P\mu\leq\dim_{qH}\mu\leq\dim_P\mu$; and if $\dim_P\mu=0$ then
(i) holds trivially, since $0\leq\dim_{qH}\mu\leq\dim_P\mu$. The last
assertion follows from Lemma~\ref{lem:floor}, see
Remark~\ref{rem:qH}(b).
\end{proof}

\subsection{Lower bounds for the packing dimension profiles}\label{sec:lower}

In this section we obtain lower bounds for the profiles $\dim_P^s\mu$ in
the regime $s>\dim_{qH}\mu$, where by Theorem~\ref{thm:threshold} they
are strictly smaller than $s$. Throughout the section, $\mu$ is a Borel
probability measure on $\R^n$, and we write
$$d_H=\dim_H\mu,\qquad P=\dim_P\mu,\qquad P_\theta=\dim_{P,\theta}\mu,\qquad N=N_\mu .$$
By \eqref{eq:chain} and Lemma~\ref{lem:compare},
$d_H\leq P_\theta\leq P\leq N\leq n$ for every $\theta\in(0,1]$.

\subsubsection{A block decomposition}

Let $0<s\leq n$. For $0\leq d<s$, $d<t<\lambda$ and $\theta\in(0,1]$ we
define the \emph{penalty}
\begin{equation}\label{eq:E}
E=E(d,t,\lambda)=\frac{(s-d)(\lambda-t)}{\lambda-d}>0,
\end{equation}
which satisfies the identity
\begin{equation}\label{eq:Eid}
E+t-s=\frac{(t-d)(\lambda-s)}{\lambda-d},
\end{equation}
and we put
\begin{equation}\label{eq:w}
w=\max\Big\{1,\ \frac{t+\theta E}{s}\Big\},\qquad \beta_\theta=\beta_\theta(d,t,\lambda)=\frac tw .
\end{equation}
Since $st/(t+\theta E)=st(\lambda-d)/\big(t(\lambda-d)+\theta(s-d)(\lambda-t)\big)$,
\begin{equation}\label{eq:beta}
\beta_\theta(d,t,\lambda)=\min\Big\{t,\ \frac{s\,t\,(\lambda-d)}{t(\lambda-d)+\theta(s-d)(\lambda-t)}\Big\}.
\end{equation}
If $\lambda\leq s$, then \eqref{eq:Eid} gives $t+\theta E\leq t+E\leq s$,
so that $w=1$ and $\beta_\theta=t$.\\

The next proposition improves \cite[Proposition 2.6 (b)]{falconer1996packing}.

\begin{proposition}\label{prop:blocks}
Let $0<s\leq n$. Let $d$ satisfy either $0<d<d_H$ or $d=d_H=0$, and
assume $d<s$. Let $\theta\in(0,1]$, let $t$ satisfy $d<t<P_\theta$ and
$t\neq s$, and let $\lambda>N$ with $\lambda\neq s$. Then for $\mu$-almost
every $x$ there is $c=c(x)>0$ such that
\[
F_s^\mu(x,h)\ \leq\ c\,h^{\beta_\theta(d,t,\lambda)}
\]
for arbitrarily small values of $h>0$.
\end{proposition}

Note that $\lambda>N\geq P\geq P_\theta>t>d$. Numbers $t$ as in the
statement exist whenever $P_\theta>d$, and this is the case if $d<d_H$,
by Lemma~\ref{lem:compare}.

\begin{proof}
By \eqref{eq:layer_profile}, $F_s^\mu(x,h)=s\,J_x(h)$, where
$J_x(h):=h^s\int_h^\infty u^{-s-1}\mu(B(x,u))\,du$, so we estimate
$J_x(h)$.

For fixed $d$ and $\lambda$ the function
$t\mapsto\beta_\theta(d,t,\lambda)$ is continuous and strictly increasing
on $(d,\lambda)$. Indeed, the second term in \eqref{eq:beta} equals
\[
\frac{s}{1+\theta\,\dfrac{s-d}{\lambda-d}\Big(\dfrac{\lambda}{t}-1\Big)} ,
\]
and the minimum of two continuous strictly increasing functions is
continuous and strictly increasing. Hence there is $t'\in(t,P_\theta)$
with $t'\neq s$ and $\beta_\theta(d,t',\lambda)>\beta_\theta(d,t,\lambda)$.
Let $w'$ be the number \eqref{eq:w} computed with $t'$, so that
$t'/w'=\beta_\theta(d,t',\lambda)$. Fix $a\in(0,1)$ with
$a\geq(\lambda-t')/(\lambda-d)$ and so close to $1$ that
$\eta=\theta(1-a)\lambda$ satisfies $(t'-\eta)/w'\geq\beta_\theta(d,t,\lambda)$.
Since $h^{e'}\leq h^e$ for $h\in(0,1)$ and $e'\geq e$, it suffices to prove
the following statement. \emph{For every admissible $t$ and every
$a\in(0,1)$ with $a\geq(\lambda-t)/(\lambda-d)$, for $\mu$-almost every
$x$ there is $c=c(x)$ such that}
\begin{equation}\label{eq:goal}
J_x(h)\leq c\,h^{(t-\eta)/w},\qquad \eta=\theta(1-a)\lambda,
\end{equation}
\emph{for arbitrarily small values of $h$.}

Now we will estimate the measure on balls centered in typical points.

By Lemma~\ref{lem:frostman},
Theorem~\ref{thm:windows}(i) and Lemma~\ref{lem:growth}, for $\mu$-almost
every $x$ the following hold simultaneously.
\begin{enumerate}
\item[(F)] $\mu(B(x,r))\leq c_1r^d$ for all $r>0$, with $c_1=c_1(x)$.
\item[(V)] There are arbitrarily small $\rho\in(0,1)$ with $\mu(B(x,r))\leq r^t$ for all $r\in[\rho,\rho^\theta]$.
\item[(D)] $\mu(B(x,r))\leq(4r/\varrho)^\lambda\mu(B(x,\varrho))$ for $0<\varrho<\rho_0(x)$ and $\varrho^a\leq r\leq1$.
\end{enumerate}
Fix such an $x$ and a number $\rho$ as in (V) with $\rho^\theta<\rho_0(x)$.
Applying (D) with $\varrho=\rho^\theta$, the top of the window, and using
$\mu(B(x,\rho^\theta))\leq\rho^{\theta t}$, which is (V) at
$r=\rho^\theta$, we obtain
\begin{enumerate}
\item[(L)] $\mu(B(x,r))\leq4^\lambda\big(r/\rho^\theta\big)^\lambda\rho^{\theta t}$ for all $r\in[\rho^{a\theta},1]$.
\end{enumerate}
The only selective condition is (V). Indeed, (F) holds at every scale and
(D) holds for every $\varrho<\rho_0(x)$ (Remark~\ref{rem:uniform}), so
(F), (V) and (L) hold simultaneously along a family of arbitrarily small
scales $\rho$, without matching two sequences of scales. In particular,
\begin{equation}\label{eq:topblock}
\mu(B(x,r))\leq\mu(B(x,\rho^{a\theta}))\leq4^\lambda\rho^{-\theta(1-a)\lambda}\rho^{\theta t}=4^\lambda\rho^{\theta t-\eta}\qquad\text{for } r\leq\rho^{a\theta}.
\end{equation}

First, suppose that $\lambda>s$. Put
\[
h=\rho^{w},\qquad R=\rho^{\theta(\lambda-t)/(\lambda-d)} .
\]
We distinguish two regimes. In \emph{regime II}, $t+\theta E\geq s$, we
have $h^s=\rho^{t+\theta E}$. In \emph{regime I}, $t+\theta E<s$, we have
$h=\rho$. Regime I can only occur if $t<s$ and $\theta<1$; indeed, $t>s$
implies $t+\theta E>s$, and for $\theta=1$ the identity \eqref{eq:Eid}
gives $t+E-s>0$ because $\lambda>s$ and $t>d$. The cut points satisfy
\[
h\ \leq\ \rho\ \leq\ \rho^\theta\ \leq\ \rho^{a\theta}\ \leq\ R\ \leq\ 1 .
\]
Here $h\leq\rho$ because $w\geq1$, $\rho^\theta\leq\rho^{a\theta}$
because $a<1$, and $\rho^{a\theta}\leq R$ is equivalent to
$a\geq(\lambda-t)/(\lambda-d)$, the same condition on $a$ as in
\cite[(2.4)]{falconer1996packing}. We shall use the inequality
\begin{equation}\label{eq:key}
h^s\rho^{\theta(t-s)}\leq\rho^t .
\end{equation}
If $t<s$ it follows from $h\leq\rho$, since
$\rho^{s+\theta(t-s)}=\rho^{t+(1-\theta)(s-t)}\leq\rho^t$. If $t>s$ we are
in regime II and $h^s\rho^{\theta(t-s)}=\rho^{t+\theta(E+t-s)}\leq\rho^t$
by \eqref{eq:Eid}. Each contribution below is of the form
$h^s\int_Iu^{-s-1}\mu(B(x,u))\,du$ over the indicated interval $I$.

We will break the integral over $5$ different blocks as follows.

\emph{Block 1, $[h,\rho]$, monotonicity from the foot of the window
(regime II only).} By (V) at $r=\rho$,
$\mu(B(x,u))\leq\mu(B(x,\rho))\leq\rho^t$ for $u\leq\rho$, and the lower
end dominates:
\[
I_1=h^s\int_h^\rho u^{-s-1}\mu(B(x,u))\,du\leq\rho^t\,h^s\int_h^\infty u^{-s-1}\,du=\frac1s\,\rho^t .
\]
In regime I we have $h=\rho$ and this block is absent.

\emph{Block 2, $[\rho,\rho^\theta]$, the window.} By (V) and $t\neq s$,
\[
I_2\leq h^s\int_\rho^{\rho^\theta}u^{t-s-1}\,du\leq\frac{h^s}{|t-s|}\max\{\rho^{t-s},\rho^{\theta(t-s)}\}\leq\frac{\rho^t}{|t-s|}.
\]
If $t<s$ the maximum is $\rho^{t-s}$ and we use $h\leq\rho$; if $t>s$ the
maximum is $\rho^{\theta(t-s)}$ and we use \eqref{eq:key}.

\emph{Block 3, $[\rho^\theta,\rho^{a\theta}]$, junction.} By
\eqref{eq:topblock} and \eqref{eq:key},
\[
I_3\leq4^\lambda\rho^{\theta t-\eta}\,h^s\int_{\rho^\theta}^\infty u^{-s-1}\,du=\frac{4^\lambda}{s}\,\rho^{-\eta}\,h^s\rho^{\theta(t-s)}\leq\frac{4^\lambda}{s}\,\rho^{t-\eta}.
\]
The loss $\rho^{-\eta}$ disappears as $a\to1$.

\emph{Block 4, $[\rho^{a\theta},R]$, the quasi-regular range.} By (L) and
$\lambda>s$ the upper end dominates:
\[
I_4\leq4^\lambda\rho^{\theta(t-\lambda)}\,h^s\int_0^Ru^{\lambda-s-1}\,du=\frac{4^\lambda}{\lambda-s}\,h^s\rho^{\theta(t-\lambda)}R^{\lambda-s}=\frac{4^\lambda}{\lambda-s}\,h^s\rho^{-\theta E},
\]
because $\theta(t-\lambda)+\theta(\lambda-t)(\lambda-s)/(\lambda-d)=-\theta(\lambda-t)(s-d)/(\lambda-d)=-\theta E$.

\emph{Block 5, $[R,\infty)$, the Frostman range.} By (F) and $d<s$ the
integral converges, and
\[
I_5\leq c_1h^s\int_R^\infty u^{d-s-1}\,du=\frac{c_1}{s-d}\,h^sR^{d-s}=\frac{c_1}{s-d}\,h^s\rho^{-\theta E}.
\]

Now, the choice of $R$ equalizes Blocks 4 and 5; it is the
solution of $\rho^{\theta(t-\lambda)}R^{\lambda-s}=R^{d-s}$. The choice of
$h$ equalizes their common value $h^s\rho^{-\theta E}$ with $\rho^t$. In
regime II, $h^s\rho^{-\theta E}=\rho^{t+\theta E-\theta E}=\rho^t$. In
regime I, $h=\rho$ and $h^s\rho^{-\theta E}=\rho^{s-\theta E}\leq\rho^t$,
because $s-\theta E>t$. Therefore, in both regimes,
\[
J_x(h)\leq I_1+\dots+I_5\leq c\,\rho^{t-\eta}=c\,h^{(t-\eta)/w},\qquad
c=\frac1s+\frac1{|t-s|}+\frac{4^\lambda}s+\frac{4^\lambda}{\lambda-s}+\frac{c_1}{s-d},
\]
since $\rho=h^{1/w}$. The constant depends on $x$ only through $c_1$ and the result follows.

Now, suppose that $\lambda<s$. 

Then
$t<P_\theta\leq N<\lambda<s$, $w=1$, and we put $h=\rho$.

\emph{Block 2, $[\rho,\rho^\theta]$.} As in the case $s<\lambda$, $I_2\leq\rho^t/(s-t)$.

\emph{Block 3, $[\rho^\theta,\rho^{a\theta}]$.} By \eqref{eq:topblock},
$I_3\leq(4^\lambda/s)\rho^{-\eta}\rho^{s+\theta(t-s)}\leq(4^\lambda/s)\rho^{t-\eta}$.

\emph{Block 4, $[\rho^{a\theta},1]$.} Now $\lambda<s$, so the lower end
dominates. By (L),
\[
I_4\leq4^\lambda\rho^{\theta(t-\lambda)}\rho^s\int_{\rho^{a\theta}}^\infty u^{\lambda-s-1}\,du=\frac{4^\lambda}{s-\lambda}\,\rho^{s+\theta(t-\lambda)+a\theta(\lambda-s)}\leq\frac{4^\lambda}{s-\lambda}\,\rho^{t},
\]
because $s+\theta(t-\lambda)+a\theta(\lambda-s)=t+(1-\theta)(s-t)+\theta(1-a)(s-\lambda)\geq t$.

\emph{Block 5, $[1,\infty)$.} Since $\mu(B(x,r))\leq1$,
$I_5\leq\rho^s\int_1^\infty u^{-s-1}\,du=\rho^s/s\leq\rho^t/s$.

Therefore $J_x(h)\leq c\,\rho^{t-\eta}=c\,h^{t-\eta}$ with
$c=1/(s-t)+4^\lambda/s+4^\lambda/(s-\lambda)+1/s$. Here neither (F) nor
the number $d$ is used.

\medskip
In both cases $\rho$ can be taken arbitrarily small, hence so can
$h=\rho^w$, and \eqref{eq:goal} is proved.

 \emph{The case $t=s$.} It is excluded only to avoid a logarithm in
Block 2. Since $t\mapsto\beta_\theta(d,t,\lambda)$ is continuous, this
exclusion costs nothing.
\end{proof}

\begin{remark}\label{rem:blocks}
\begin{enumerate}[(a)]
\item \emph{The case $\theta=1$.} Then $\rho^\theta=\rho$, Block 2 is
empty, regime I does not occur, $h=\rho^{(t+E)/s}$ and
$R=\rho^{(\lambda-t)/(\lambda-d)}$. With $\lambda$ in the role of
$n(1+\varepsilon)$, these are the partition $[h,\rho]$, $[\rho,\rho^a]$,
$[\rho^a,R]$, $[R,\infty)$, the scales and the exponent $st/(t+E)$ of the
proof of \cite[Proposition 2.6(a)]{falconer1996packing}. Letting
$\lambda\downarrow n$ recovers \cite[Proposition 2.6(a)]{falconer1996packing},
whose exponent is $st/\big(t+(s-d)(n-t)/(n-d)\big)$ (for an integer
$s=m$).
\item Compared with
\cite{falconer1996packing}, with $N$ in place of $n$, the only changes in
the exponent are the factor $\theta$ in front of the penalty $E$ and the
cap $t$. Above the packing scale $\rho$ the only information available in
\cite{falconer1996packing} is the quasi-Lebesgue bound (D), valid from
$\rho^a$ on, which grows fast. The window postpones the quasi-regular
regime from $\rho$ to $\rho^\theta$, since between $\rho$ and
$\rho^\theta$ the measure grows at most like $r^t$. This is why the
relevant quantity is the measure at the foot of the window, and why the
estimates are parametrized by $\rho$ and not by $\rho^\theta$.
\end{enumerate}
\end{remark}

\subsubsection{The lower bound}

For $\lambda>d$, $\theta\in(0,1]$ and $0<t\leq\lambda$ we write
\begin{equation}\label{eq:G}
G_\lambda(\theta,t)=\frac{s\,t\,(\lambda-d)}{t(\lambda-d)+\theta(s-d)(\lambda-t)},
\end{equation}
where $s$ and $d$ will always be clear from the context, so that
$\beta_\theta(d,t,\lambda)=\min\{t,G_\lambda(\theta,t)\}$. From now on we
take $d=d_H$. For $d_H<s\leq n$ we also write
\begin{equation}\label{eq:alphaFM}
\alpha_{FM}(s)=\frac{P\,(1-d_H/n)}{1+(1/s-1/n)P-d_H/s};
\end{equation}
for an integer $s=m$ this is the lower bound \eqref{eq:intro_FM} of
\cite[Theorem 3.3]{falconer1996packing} for $\dim_P\mu_V$.

\begin{lemma}\label{lem:algebra}
Let $d_H<s\leq n$, $\lambda>d_H$, $\theta\in(0,1]$ and $0<t\leq\lambda$.
\begin{enumerate}[(i)]
\item $\displaystyle\frac{1}{G_\lambda(\theta,t)}=\frac1s+\frac{\theta(s-d_H)}{s(\lambda-d_H)}\Big(\frac\lambda t-1\Big)$. In particular $G_\lambda(\theta,t)$ is increasing in $t$ and nonincreasing in $\theta$.
\item $G_\lambda(1,t)\leq t$ if and only if $(\lambda-s)(t-d_H)\geq0$,
and the same holds with strict inequalities. If $\lambda\leq s$ and
$t\geq d_H$, then $G_\lambda(\theta,t)\geq t$ for every $\theta$.
\item If $P\leq\lambda$, then $G_\lambda(1,P)=\dfrac{P\,(1-d_H/\lambda)}{1+(1/s-1/\lambda)P-d_H/s}$; in particular $G_n(1,P)=\alpha_{FM}(s)$.
\item The function $\lambda\mapsto G_\lambda(1,P)$ is nonincreasing on
$[P,\infty)$, and strictly decreasing if $d_H<P$.
\item If $\lambda\geq\max\{s,P\}$, then $d\mapsto G_\lambda(1,P)$,
computed with $d$ in place of $d_H$, is nondecreasing on $[0,s)$.
\end{enumerate}
\end{lemma}

\begin{proof}
(i) Divide the denominator of \eqref{eq:G} by $st(\lambda-d_H)$.
(ii) By (i), $G_\lambda(1,t)\leq t$ if and only if
$t(\lambda-d_H)+(s-d_H)(\lambda-t)\geq s(\lambda-d_H)$, and the
difference of the two sides equals $(\lambda-s)(t-d_H)$. If $\lambda\leq s$
and $t\geq d_H$ this gives $G_\lambda(1,t)\geq t$, and
$G_\lambda(\theta,t)\geq G_\lambda(1,t)$ by (i). (iii) Divide numerator
and denominator of $G_\lambda(1,P)$ by $s\lambda$ and use
$P(\lambda-d_H)+(s-d_H)(\lambda-P)=\lambda s+\lambda P-sP-\lambda d_H$.
(iv) By (i),
$1/G_\lambda(1,P)=1/s+(s-d_H)(\lambda-P)/(sP(\lambda-d_H))$, and
$\lambda\mapsto(\lambda-P)/(\lambda-d_H)$ has derivative
$(P-d_H)/(\lambda-d_H)^2\geq0$. (v) Similarly, $d\mapsto(s-d)/(\lambda-d)$
has derivative $(s-\lambda)/(\lambda-d)^2\leq0$.
\end{proof}

\begin{theorem}\label{thm:FMimproved}
Let $\mu$ be a Borel probability measure on $\R^n$ and let $0<s\leq n$
with $\dim_H\mu<s$. Then
\begin{equation}\label{eq:FMimproved}
\dim_P^s\mu\ \geq\ \sup_{0<\theta\leq1}\Lambda_\theta,\qquad
\Lambda_\theta=\min\Big\{P_\theta,\ \frac{s\,P_\theta\,(N-d_H)}{P_\theta(N-d_H)+\theta(s-d_H)(N-P_\theta)}\Big\},
\end{equation}
where $d_H=\dim_H\mu$, $P_\theta=\dim_{P,\theta}\mu$ and $N=N_\mu$. If
$N=d_H$, then $P_\theta=d_H$ for all $\theta$ and $\Lambda_\theta$ is to
be read as $d_H$.
\end{theorem}

\begin{proof}
In the notation \eqref{eq:G} with $d=d_H$,
$\Lambda_\theta=\min\{P_\theta,G_N(\theta,P_\theta)\}$. Recall that
$d_H\leq P_\theta\leq P\leq N$.

First suppose that there exists $\theta$ with $P_\theta=d_H$. By
Theorem~\ref{thm:threshold} and Lemma~\ref{lem:compare},
$\dim_P^s\mu\geq\min\{s,\dim_{qH}\mu\}\geq\min\{s,d_H\}=d_H$. If
$P_\theta=d_H<N$, then
$G_N(\theta,d_H)=sd_H/(d_H+\theta(s-d_H))\geq d_H$, so that
$\Lambda_\theta=d_H$, and if $N=d_H$ then $\Lambda_\theta=d_H$ by
convention. Hence these values of $\theta$ are covered.

Now, fix $\theta$ with $P_\theta>d_H$. Choose $d$ with
$0<d<d_H$, or $d=0$ if $d_H=0$; choose $t$ with $d<t<P_\theta$ and
$t\neq s$; and choose $\lambda>N$ with $\lambda\neq s$, and with
$\lambda<s$ if $N<s$. By Proposition~\ref{prop:blocks}, for $\mu$-almost
every $x$ there are $c(x)>0$ and $h_k\downarrow0$ with
$F_s^\mu(x,h_k)\leq c(x)\,h_k^{\beta_\theta(d,t,\lambda)}$. Hence
$\liminf_{h\to0}h^{-\beta'}F_s^\mu(x,h)=0$ for every
$\beta'<\beta_\theta(d,t,\lambda)$, and therefore
$\dim_P^s\mu\geq\beta_\theta(d,t,\lambda)$.

Now, if $N<s$, then
$\beta_\theta(d,t,\lambda)=t$ (see after \eqref{eq:beta}); letting
$t\uparrow P_\theta$ gives $\dim_P^s\mu\geq P_\theta$, and
$P_\theta=\Lambda_\theta$ by Lemma~\ref{lem:algebra}(ii). If $N\geq s$,
we let $d\uparrow d_H$, $t\uparrow P_\theta$ and $\lambda\downarrow N$
along sequences with $s<\lambda$, $d<t$ and $t\neq s$. The function
$(d,t,\lambda)\mapsto\beta_\theta(d,t,\lambda)$ is continuous at
$(d_H,P_\theta,N)$, because the denominator
$t(\lambda-d)+\theta(s-d)(\lambda-t)$ is positive there (note that
$N\geq P_\theta>d_H$), and its value at this point is $\Lambda_\theta$.
Hence $\dim_P^s\mu\geq\Lambda_\theta$.

Finally, taking the supremum over $\theta\in(0,1]$ completes the
proof.
\end{proof}

\subsection{Comparison with the bound of Falconer and Mattila}

In this subsection we compare the lower bound of Theorem~\ref{thm:FMimproved}
with the bound $\alpha_{FM}(s)$ of \eqref{eq:alphaFM}, that is, the bound
\eqref{eq:intro_FM} of Falconer and Mattila
\cite[Theorem 3.3]{falconer1996packing} extended to real exponents $s$.
We use the notation of Theorem~\ref{thm:FMimproved} and assume that
$d_H<s<n$ and $d_H<P$; if $d_H=P$, both bounds are equal to $d_H$.
Our bound is never worse than $\alpha_{FM}(s)$, and we determine exactly
when it is strictly better. If $N<n$, that is, if $\supp\mu$ is compact
and $\dim_A(\supp\mu)<n$, then the term $\theta=1$ in
\eqref{eq:FMimproved} alone already improves $\alpha_{FM}(s)$ strictly
(Proposition~\ref{prop:theta1}). If $N=n$, for instance if
$\dim_A(\supp\mu)=n$, the term $\theta=1$ coincides with
$\alpha_{FM}(s)$, and the bound \eqref{eq:FMimproved} improves
$\alpha_{FM}(s)$ strictly if and only if $\dim_{P,\theta}\mu$ is strictly
larger than the universal lower bound $\sigma_\theta(\dim_P\mu)$ of
Lemma~\ref{lem:floor} for some $\theta$ in an explicit interval
$[\theta^\dagger,1)$ (Corollary~\ref{cor:floor_improvement}). In this
case the improvement comes entirely from the packing intermediate
dimensions with $\theta<1$.

In Propositions~\ref{prop:theta1} and~\ref{prop:criterion},
Corollary~\ref{cor:floor_improvement} we
assume that $d_H<s<n$, $d_H<P$ and $d_H<N$; if $d_H=P$, then
$P_\theta=d_H$ for all $\theta$ and all the quantities involved are
equal to $d_H$.

\begin{proposition}\label{prop:theta1}
\begin{enumerate}[(a)]
\item If $N\geq s$, then
$\Lambda_1=G_N(1,P)=\dfrac{P\,(1-d_H/N)}{1+(1/s-1/N)P-d_H/s}$, that is,
the bound \eqref{eq:alphaFM} with $n$ replaced by $N$. If $N\leq s$, then
$\Lambda_1=P$.
\item $\Lambda_1\geq\alpha_{FM}(s)$, with equality if and only if $N=n$.
In particular, if $N=n$ then Theorem~\ref{thm:FMimproved} contains the
bound of \cite[Theorem 3.3]{falconer1996packing}, extended to real
exponents, and if $N<n$ it improves it strictly.
\item $\Lambda_1\geq P-\max\{0,N-s\}$, and the inequality is strict if
$s<N$ and $P<N$.
\end{enumerate}
\end{proposition}

\begin{proof}
(a) Since $P_1=P$ by Remark~\ref{rem:basic}(c),
$\Lambda_1=\min\{P,G_N(1,P)\}$. If $N\geq s$,
Lemma~\ref{lem:algebra}(ii) gives $G_N(1,P)\leq P$, and (iii) gives the
second expression. If $N\leq s$, Lemma~\ref{lem:algebra}(ii) gives
$G_N(1,P)\geq P$.

(b) If $N\geq s$, then $\Lambda_1=G_N(1,P)\geq G_n(1,P)=\alpha_{FM}(s)$ by
Lemma~\ref{lem:algebra}(iii) and (iv), with strict inequality when $N<n$
since $d_H<P$. If $N<s$, then $\Lambda_1=P>G_n(1,P)=\alpha_{FM}(s)$ by
Lemma~\ref{lem:algebra}(ii), because $n>s$ and $P>d_H$.

(c) If $N\leq s$ this is (a). Let $N>s$. By Lemma~\ref{lem:algebra}(v),
$G_N(1,P)$ is at least its value for $d_H=0$, which is
$sPN/(PN+s(N-P))$. A direct computation gives
\[
\frac{sPN}{PN+s(N-P)}-\big(P-(N-s)\big)=\frac{(N-P)(P+s)(N-s)}{PN+s(N-P)}\ \geq\ 0,
\]
with strict inequality if $P<N$.
\end{proof}

As a consequence, we obtain the following lower bound, in which the
measure enters only through its packing dimension, so that the size of
the possible loss is governed by the support alone and is common to all
measures supported on a given set.

\begin{corollary}\label{cor:assouad}
Let $\mu$ be a Borel probability measure on $\R^n$ and let $0<s\leq n$.
Then
\begin{equation}\label{eq:assouad_bound_profile}
\dim_P^s\mu\ \geq\ \dim_P\mu-\max\{0,\ N_\mu-s\}.
\end{equation}
In particular, if $\supp\mu$ is compact, then
$\dim_P^s\mu\geq\dim_P\mu-\max\{0,\dim_A(\supp\mu)-s\}$, and
$\dim_P^s\mu=\dim_P\mu$ whenever $s\geq\dim_A(\supp\mu)$.
\end{corollary}

\begin{proof}
If $s=n$, \eqref{eq:assouad_bound_profile} follows from \eqref{eq:FHcor3}.
Assume $s<n$. If $d_H\geq s$, then $N\geq P\geq d_H\geq s$, and
Theorem~\ref{thm:threshold} and Lemma~\ref{lem:compare} give
$\dim_P^s\mu=s\geq P-(N-s)$. If $d_H<s$ and $d_H=P$, then
$\dim_P^s\mu\geq\min\{s,\dim_{qH}\mu\}\geq d_H=P$ by the same results.
Finally, if $d_H<s$ and $d_H<P\leq N$, then Theorem~\ref{thm:FMimproved}
and Proposition~\ref{prop:theta1}(c) give
$\dim_P^s\mu\geq\Lambda_1\geq P-\max\{0,N-s\}$. The last assertion
follows from \eqref{eq:assouad_bound_profile} and
Remark~\ref{rem:profiles}(c).
\end{proof}

Evaluating the potential $F^\mu_s(x,\cdot)$ at the packing scale $\rho$
itself, instead of at the smaller scale $h=\rho^w$ of
Proposition~\ref{prop:blocks}, one obtains directly the bound
\eqref{eq:assouad_bound_profile}; by Proposition~\ref{prop:theta1}(c),
the term $\theta=1$ of Theorem~\ref{thm:FMimproved} is better, strictly
so when $s<N$ and $P<N$. For $d_H=0$ this term equals
$P/\bigl(1+(1/s-1/N)P\bigr)$, which for $N=n$ is the bound obtained by
taking $\dim_H\mu=0$ in \eqref{eq:intro_FM}.

\begin{remark}\label{rem:FFS}
Falconer, Fraser and Shmerkin \cite[Theorem 2.2]{falconer2021assouad}
proved that if $F\subset\R^n$ is a Borel set, $s\in(0,n]$ and
$\theta\in(0,1)$, then
$$\dim_P^{s}F\ \geq\ \dim_P F-\max\bigl\{0,\ \overline{\dim}_A^{\theta}F-s,\
(\dim_A F-s)(1-\theta)\bigr\},$$
where $\overline{\dim}_A^{\theta}F\leq\dim_A F$ denotes the upper
Assouad spectrum \cite{fraser2020assouad}. Since the maximum above never
exceeds $\max\{0,\dim_A F-s\}$, this gives
\begin{equation}\label{eq:FFS_assouad_form}
\dim_P^{s}F\ \geq\ \dim_P F-\max\{0,\ \dim_A F-s\},
\end{equation}
which we refer to as the \emph{Assouad dimension form} of their theorem.
In \cite{falconer2021assouad} the profiles of sets are defined through
capacities; for integer exponents this definition coincides with the one
of Section~\ref{sec:prelim}, since both quantities equal the packing
dimension of $\pi_V F$ for $\gamma_{n,m}$-almost every $V\in G(n,m)$, see
\cite[Theorem 10]{falconer1997packing} and
\cite[Theorem 2.1]{falconer2021assouad}.

Corollary~\ref{cor:assouad} is not an adaptation of their argument, which
relies on the capacity formulation of the box dimension profiles, in which
the profile of a set is defined through an infimum over all measures
supported on it. This allows one to choose, at each scale $r$, a measure
adapted to that scale, and their proof takes it to be the normalized
counting measure on a maximal $r$-separated subset of $F$, the packing
case being deduced afterwards by a Baire category argument. In the present
setting no such freedom is available: the measure $\mu$ is given and is
the same at every scale, and the profile is defined through the behaviour
of its own potentials at $\mu$-almost every point. Our argument relies
instead on the regular measure of Lemma~\ref{lem:regular}, which replaces
the Lebesgue measure in the argument of Falconer and Mattila.

On the other hand, Corollary~\ref{cor:assouad} yields
\eqref{eq:FFS_assouad_form} for analytic sets. Indeed, let $E\subset\R^n$
be an analytic set. For every $\mu\in\mathcal{M}_c^+(E)$ we have
$\dim_A(\supp\mu)\leq\dim_A E$, so that
$$\dim_P^s\mu\geq\dim_P\mu-\max\{0,\ \dim_A E-s\},$$
and taking the supremum over $\mu\in\mathcal{M}_c^+(E)$, together with
\eqref{eq:dimP_sets} and the definition of $\dim_P^s E$, gives
$\dim_P^s E\geq\dim_P E-\max\{0,\ \dim_A E-s\}$.
\end{remark}

We now turn to the terms $\theta<1$ in \eqref{eq:FMimproved}, which are
the ones that require the packing intermediate dimensions. Recall the
function $\sigma_\theta(P)=\theta NP/(N-(1-\theta)P)$ of
Lemma~\ref{lem:floor}.

\begin{proposition}\label{prop:criterion}
Let $\theta\in(0,1)$ and $t\in(0,N]$.
\begin{enumerate}[(a)]
\item $G_N(\theta,t)\geq G_N(1,P)$ if and only if $t\geq\sigma_\theta(P)$,
and the same holds with strict inequalities and with equalities. The
threshold $\sigma_\theta(P)$ depends neither on $s$ nor on $d_H$.
\item Assume $s<N$ and $P<N$, and put $G=G_N(1,P)=\Lambda_1$. Then
$d_H<G<P$, the function $\theta\mapsto\sigma_\theta(P)$ is increasing
with $\sigma_1(P)=P$, and the number
\[
\theta^\dagger=\frac{G\,(N-P)}{P\,(N-G)}\in(0,1)
\]
satisfies $\sigma_{\theta^\dagger}(P)=G$. Moreover $\Lambda_\theta>\Lambda_1$
if and only if $P_\theta>\max\{\sigma_\theta(P),G\}$, and for
$\theta\in[\theta^\dagger,1)$ this condition reduces to
$P_\theta>\sigma_\theta(P)$.
\item If $N\leq s$ or $P=N$, then $\Lambda_\theta\leq\Lambda_1$ for every
$\theta\in(0,1]$.
\end{enumerate}
\end{proposition}

\begin{proof}
(a) By Lemma~\ref{lem:algebra}(i), after cancelling the positive factor
$(s-d_H)/(s(N-d_H))$, the inequality $G_N(\theta,t)\geq G_N(1,P)$ is
equivalent to $\theta(N/t-1)\leq N/P-1$, that is, to
$\theta P(N-t)\leq t(N-P)$, that is, to $t(N-P+\theta P)\geq\theta NP$.
The strict version and the version with equalities are identical.

(b) $G<P$ follows from Lemma~\ref{lem:algebra}(ii), since $N>s$ and
$P>d_H$, and $G>d_H$ follows from Lemma~\ref{lem:algebra}(i) and
$G_N(1,d_H)=d_H$. The derivative of $\theta\mapsto\sigma_\theta(P)$ is
$NP(N-P)/(N-P+\theta P)^2>0$. Solving $\theta NP=G(N-P+\theta P)$ gives
$\theta^\dagger$, and $0<\theta^\dagger<1$ because $0<G<P$. Finally,
$\Lambda_\theta>G$ means $P_\theta>G$ and $G_N(\theta,P_\theta)>G$, and
by (a) the latter means $P_\theta>\sigma_\theta(P)$. For
$\theta\geq\theta^\dagger$ we have $\sigma_\theta(P)\geq G$.

(c) If $N\leq s$, then $\Lambda_\theta=P_\theta\leq P=\Lambda_1$ by
Lemma~\ref{lem:algebra}(ii) and Proposition~\ref{prop:theta1}(a). If
$P=N$, then $\sigma_\theta(P)=P\geq P_\theta$, and (a) gives
$G_N(\theta,P_\theta)\leq G_N(1,P)$, so that
$\Lambda_\theta\leq\min\{P,G_N(1,P)\}=\Lambda_1$.
\end{proof}

Combining Proposition~\ref{prop:criterion} with the universal lower bound
of Lemma~\ref{lem:floor}, we obtain the following description of the
improvement provided by the packing intermediate dimensions.

\begin{corollary}\label{cor:floor_improvement}
Assume $s<N$ and $P<N$, and let $G=\Lambda_1$ and $\theta^\dagger$ be as
in Proposition~\ref{prop:criterion}(b).
\begin{enumerate}[(i)]
\item For every $\theta\in[\theta^\dagger,1]$ we have
$\Lambda_\theta\geq\Lambda_1$. Hence on this range the terms $\theta<1$
in \eqref{eq:FMimproved} never fall below the term $\theta=1$, and in
particular never below $\alpha_{FM}(s)$.
\item $\sup_\theta\Lambda_\theta>\Lambda_1$ if and only if there is
$\theta\in[\theta^\dagger,1)$ with
$\dim_{P,\theta}\mu>\sigma_\theta(\dim_P\mu)$, that is, if and only if
$\dim_{P,\theta}\mu$ exceeds strictly the universal lower bound of
Lemma~\ref{lem:floor} for some $\theta\in[\theta^\dagger,1)$.
\item In the situation of (ii), Theorem~\ref{thm:FMimproved} improves
strictly the bound $\alpha_{FM}(s)$, also when $N=n$.
\end{enumerate}
\end{corollary}

\begin{proof}
(i) By Lemma~\ref{lem:floor} and Proposition~\ref{prop:criterion}(b),
$P_\theta\geq\sigma_\theta(P)\geq\sigma_{\theta^\dagger}(P)=G$, and by
Proposition~\ref{prop:criterion}(a), $G_N(\theta,P_\theta)\geq G$. Hence
$\Lambda_\theta\geq G$. The last assertion follows from
Proposition~\ref{prop:theta1}(b).

(ii) If $P_\theta>\sigma_\theta(P)$ for some $\theta\in[\theta^\dagger,1)$,
then $\Lambda_\theta>\Lambda_1$ by Proposition~\ref{prop:criterion}(b).
Conversely, let $\Lambda_\theta>\Lambda_1$ for some $\theta$. Then
$P_\theta>\max\{\sigma_\theta(P),G\}$, and $\theta<1$ since
$P_1=P=\sigma_1(P)$. If $\theta\geq\theta^\dagger$ we are done. If
$\theta<\theta^\dagger$, then by monotonicity
$P_{\theta^\dagger}\geq P_\theta>G=\sigma_{\theta^\dagger}(P)$, so that
$\theta^\dagger$ has the required property.

(iii) This follows from (ii) and $\Lambda_1\geq\alpha_{FM}(s)$.
\end{proof}

\section{Packing dimension of H\"{o}lder images}\label{sec:images}

In this section we apply the results of Sections~\ref{sec:threshold}
and~\ref{sec:lower} to the images of measures under random H\"{o}lder
maps, a family which includes fractional Brownian motion and orthogonal
projections as particular cases. The link between images and profiles is
given by Theorem~\ref{thm:images}, which contains \eqref{eq:intro_FH} and
\eqref{eq:intro_xiao} as particular cases; its proof only involves the
profiles $\dim_P^s\mu$, through the elementary estimates of
Propositions~\ref{prop:holder_upper} and~\ref{prop:random_lower}.

\begin{definition}\label{def:random_holder}
Let $(\Omega,\mathcal{F},\mathbb{P})$ be a probability space, let
$K\subset\R^n$ be a Borel set, $m\in\mathbb{N}$ and $0<\alpha\leq1$. A
family $\{f_\omega:K\to\R^m\}_{\omega\in\Omega}$ of continuous maps such
that $(\omega,x)\mapsto f_\omega(x)$ is
$\mathcal{F}\otimes\mathcal{B}(K)$-measurable is a \emph{random
$\alpha$-H\"{o}lder map} from $K$ to $\R^m$ if the following two
conditions hold:
\begin{enumerate}
\item[(H1)] for $\mathbb{P}$-almost every $\omega$ and every
$\beta\in(0,\alpha)$, $f_\omega$ is $\beta$-H\"{o}lder continuous on every
compact subset of $K$;
\item[(H2)] there exists $c>0$ such that, for all $x,y\in K$ and $r>0$,
$$\mathbb{P}\bigl(\{\omega\,:\,|f_\omega(x)-f_\omega(y)|\leq r\}\bigr)\leq
c\,\min\bigl\{1,\ r^{m}|x-y|^{-\alpha m}\bigr\}.$$
\end{enumerate}
\end{definition}

Condition (H1) is used to bound the dimension of the images from above,
and (H2), a non-degeneracy condition, to bound it from below. Note that
the right-hand side of (H2) is, up to the constant, the kernel of the
$\alpha m$-dimensional profile at the scale $r^{1/\alpha}$, since
$\min\{1,r^m|x-y|^{-\alpha m}\}=\min\{1,(r^{1/\alpha})^{\alpha m}|x-y|^{-\alpha m}\}$.

\begin{example}[Fractional Brownian motion]\label{ex:fbm}
Let $0<\alpha<1$ and let $X:\R^n\to\R^m$ be an index-$\alpha$ fractional
Brownian motion defined on $(\Omega,\mathcal{F},\mathbb{P})$. Condition (H1) is the H\"{o}lder continuity of the sample
paths recalled in Section~\ref{sec:prelim}. 
Moreover, it was proved in \cite[Theorem 3.4]{Burrell} that $X$ satisfies (H2). Hence $X$ is a random
$\alpha$-H\"{o}lder map from $\R^n$ to $\R^m$.
\end{example}

\begin{example}[Orthogonal projections]\label{ex:projections}
Let $1\leq m\leq n$. Its well known that the set of $m-$planes in $\mathbb{R}^n$, $G(n,m)$, endowed with its Haar measure $\gamma_{n,m}$ satisfies (H2) with $\alpha=1$, see for example
\cite[Chapter 3]{mattila1999geometry}. Since any orthogonal projection is $1$-Lipschitz, we have that the family $\{P_V\}_{V\in G(n,m)}$ is a random $1$-H\"{o}lder
map from $\R^n$ to $\R^m$.
\end{example}

We now establish the two estimates on which the identity of
Theorem~\ref{thm:images} relies. The first one is deterministic.

\begin{proposition}\label{prop:holder_upper}
Let $\mu$ be a Borel probability measure on $\R^n$, let $m\in\mathbb{N}$
and $0<\beta\leq1$, and let $f:\supp\mu\to\R^m$ be $\beta$-H\"{o}lder
continuous. Then
$$\dim_P\mu_f\ \leq\ \frac{1}{\beta}\,\dim_P^{\beta m}\mu .$$
Consequently, if $0<\alpha\leq1$ and $f$ is $\beta$-H\"{o}lder continuous
for every $\beta\in(0,\alpha)$, then
$\dim_P\mu_f\leq\frac{1}{\alpha}\dim_P^{\alpha m}\mu$.
\end{proposition}

\begin{proof}
Let $c>0$ be such that $|f(x)-f(y)|\leq c\,|x-y|^{\beta}$ for all
$x,y\in\supp\mu$. Then, for $x,y\in\supp\mu$ and $r>0$,
$$\min\bigl\{1,\ r^m|x-y|^{-\beta m}\bigr\}\leq\min\bigl\{1,\ c^m r^m|f(x)-f(y)|^{-m}\bigr\}
\leq C\min\bigl\{1,\ r^m|f(x)-f(y)|^{-m}\bigr\},$$
with $C=\max\{1,c^m\}$, and integrating with respect to $\mu(y)$ we
obtain, by the definition of the image measure,
\begin{equation}\label{eq:holder_kernels}
F^\mu_{\beta m}\bigl(x,r^{1/\beta}\bigr)\leq C\,F^{\mu_f}_m\bigl(f(x),r\bigr)\qquad\text{for every }x\in\supp\mu .
\end{equation}
Let $u>\dim_P^{\beta m}\mu/\beta$. Since $\beta u>\dim_P^{\beta m}\mu$,
the set
$$A:=\Bigl\{x\in\supp\mu\,:\,\liminf_{\rho\to0}\rho^{-\beta u}F^\mu_{\beta m}(x,\rho)>0\Bigr\}$$
has positive $\mu$-measure, and for $x\in A$, by
\eqref{eq:holder_kernels},
\begin{align*}
\liminf_{r\to0}r^{-u}F^{\mu_f}_m\bigl(f(x),r\bigr)&\geq C^{-1}\liminf_{r\to0}r^{-u}F^\mu_{\beta m}\bigl(x,r^{1/\beta}\bigr)\\
&=C^{-1}\liminf_{\rho\to0}\rho^{-\beta u}F^\mu_{\beta m}(x,\rho)>0 .
\end{align*}
Hence
$A\subset f^{-1}\bigl(\{w:\liminf_{r\to0}r^{-u}F^{\mu_f}_m(w,r)>0\}\bigr)$,
the latter set has positive $\mu_f$-measure, and by
Remark~\ref{rem:profiles}(a), $\dim_P^m\mu_f\leq u$. As $\mu_f$ is a
measure on $\R^m$, \eqref{eq:FHcor3} gives $\dim_P\mu_f=\dim_P^m\mu_f\leq u$,
and since $u>\dim_P^{\beta m}\mu/\beta$ was arbitrary, the first claim
follows. If $f$ is $\beta$-H\"{o}lder continuous for every
$\beta\in(0,\alpha)$, then by Remark~\ref{rem:profiles}(b),
$\dim_P\mu_f\leq\dim_P^{\beta m}\mu/\beta\leq\dim_P^{\alpha m}\mu/\beta$
for every $\beta<\alpha$, and letting $\beta\to\alpha$ completes the
proof.
\end{proof}

The second estimate is probabilistic, and only uses condition (H2).

\begin{proposition}\label{prop:random_lower}
Let $(\Omega,\mathcal{F},\mathbb{P})$ be a probability space, $\mu$ a
Borel probability measure on $\R^n$, $m\in\mathbb{N}$ and
$0<\alpha\leq1$. Let $\{f_\omega:\supp\mu\to\R^m\}_{\omega\in\Omega}$ be
a family of continuous maps such that $(\omega,x)\mapsto f_\omega(x)$ is
$\mathcal{F}\otimes\mathcal{B}(\supp\mu)$-measurable and (H2) holds with
$K=\supp\mu$. Then
$$\dim_P\mu_{f_\omega}\ \geq\ \frac{1}{\alpha}\,\dim_P^{\alpha m}\mu\qquad\text{for }\mathbb{P}\text{-almost every }\omega\in\Omega.$$
\end{proposition}

\begin{proof}
Let $0<t<\dim_P^{\alpha m}\mu$. By Remark~\ref{rem:profiles}(a), for
$\mu$-a.e.\ $x$ and every $\varepsilon>0$ there exists a sequence
$\rho_k\to0$ such that $F^\mu_{\alpha m}(x,\rho_k)<\varepsilon\,\rho_k^{t}$.
Fix $x$ in the corresponding full-measure set and put
$r_k:=\rho_k^{\alpha}$. By Fubini's theorem and (H2), for every $r>0$,
\begin{align*}
\int r^{-t/\alpha}\mu_{f_\omega}\bigl(B(f_\omega(x),r)\bigr)\,d\mathbb{P}(\omega)
&=r^{-t/\alpha}\int\mathbb{P}\bigl(\{\omega:|f_\omega(x)-f_\omega(y)|\leq r\}\bigr)\,d\mu(y)\\
&\leq c\,r^{-t/\alpha}\int\min\bigl\{1,\ r^m|x-y|^{-\alpha m}\bigr\}\,d\mu(y)\\
&=c\,r^{-t/\alpha}F^\mu_{\alpha m}\bigl(x,r^{1/\alpha}\bigr).
\end{align*}
Along the sequence $\{r_k\}$ the right-hand side is at most
$c\,\varepsilon$; hence, by Fatou's lemma,
$$\int\liminf_{r\to0}\, r^{-t/\alpha}\mu_{f_\omega}\bigl(B(f_\omega(x),r)\bigr)\,d\mathbb{P}(\omega)\leq c\,\varepsilon.$$
Note that the left-hand side no longer involves the sequence $\{r_k\}$,
which depended on $x$ and $\varepsilon$. Since $\varepsilon>0$ was
arbitrary, for $\mathbb{P}$-a.e.\ $\omega$ we have
$\liminf_{r\to0}r^{-t/\alpha}\mu_{f_\omega}(B(f_\omega(x),r))=0$. This
holds for $\mu$-a.e.\ $x$, so by Fubini's theorem, for $\mathbb{P}$-a.e.\
$\omega$ it holds for $\mu$-a.e.\ $x$. For such an $\omega$, the set
$Z_\omega$ of points $w\in\R^m$ with
$\liminf_{r\to0}r^{-t/\alpha}\mu_{f_\omega}(B(w,r))=0$ satisfies
$\mu_{f_\omega}(Z_\omega)=\mu(f_\omega^{-1}(Z_\omega))=1$, and therefore
$t/\alpha\leq\dim_P\mu_{f_\omega}$ by \eqref{eq:hu}. Letting
$t\to\dim_P^{\alpha m}\mu$ along a countable sequence and intersecting
the corresponding almost sure events completes the proof.
\end{proof}

\begin{theorem}\label{thm:images}
Let $\mu$ be a Borel probability measure on $\R^n$ with compact support,
let $m\in\mathbb{N}$ and $0<\alpha\leq1$, and let $\{f_\omega\}$ be a
random $\alpha$-H\"{o}lder map from $\supp\mu$ to $\R^m$. Then
$$\dim_P\mu_{f_\omega}=\frac{1}{\alpha}\,\dim_P^{\alpha m}\mu\qquad\text{for }\mathbb{P}\text{-almost every }\omega\in\Omega .$$
\end{theorem}

\begin{proof}
By (H1), for $\mathbb{P}$-almost every $\omega$ the map $f_\omega$ is
$\beta$-H\"{o}lder continuous on the compact set $\supp\mu$ for every
$\beta\in(0,\alpha)$, and Proposition~\ref{prop:holder_upper} gives
$\dim_P\mu_{f_\omega}\leq\dim_P^{\alpha m}\mu/\alpha$. The reverse
inequality holds almost surely by Proposition~\ref{prop:random_lower}.
\end{proof}

By Examples~\ref{ex:fbm} and~\ref{ex:projections},
Theorem~\ref{thm:images} contains Xiao's formula \eqref{eq:intro_xiao}
and the projection theorem \eqref{eq:intro_FH} of Falconer and Howroyd
for compactly supported measures; for these two families the identity
holds for every finite Borel measure, by \cite[Theorem 3.1]{xiao1997packing}
and \cite[Theorem 6]{falconer1997packing}. Combining
Theorem~\ref{thm:images} with the results of Sections~\ref{sec:threshold}
and~\ref{sec:lower}, we obtain the following consequences. The first one
is the threshold characterization of full packing dimension for images.

\begin{corollary}\label{cor:images_threshold}
Let $\mu$ be a Borel probability measure on $\R^n$, $m\in\mathbb{N}$ and
$0<\alpha\leq1$.
\begin{enumerate}[(a)]
\item If $f:\supp\mu\to\R^m$ is $\beta$-H\"{o}lder continuous for every
$\beta\in(0,\alpha)$ and $\dim_P\mu_f=m$, then $\alpha m\leq\dim_{qH}\mu$.
\item If $\supp\mu$ is compact and $\{f_\omega\}$ is a random
$\alpha$-H\"{o}lder map from $\supp\mu$ to $\R^m$, then
$$\dim_P\mu_{f_\omega}=m\ \text{ almost surely}\quad\Longleftrightarrow\quad\alpha m\leq\dim_{qH}\mu ,$$
and if $\alpha m>\dim_{qH}\mu$, then $\dim_P\mu_{f_\omega}<m$ almost
surely.
\end{enumerate}
\end{corollary}

\begin{proof}
(a) By Proposition~\ref{prop:holder_upper} and
Remark~\ref{rem:profiles}(c),
$m=\dim_P\mu_f\leq\dim_P^{\alpha m}\mu/\alpha\leq m$, so that
$\dim_P^{\alpha m}\mu=\alpha m$. Since $\dim_P^{\alpha m}\mu\leq\dim_P\mu\leq n$,
we have $\alpha m\leq n$, and \eqref{eq:threshold} gives
$\alpha m\leq\dim_{qH}\mu$.

(b) By Theorem~\ref{thm:images},
$\dim_P\mu_{f_\omega}=\dim_P^{\alpha m}\mu/\alpha$ almost surely. If
$\alpha m\leq\dim_{qH}\mu$, then $\alpha m\leq n$ and
\eqref{eq:threshold} gives $\dim_P^{\alpha m}\mu=\alpha m$, so that
$\dim_P\mu_{f_\omega}=m$ almost surely. If $\alpha m>\dim_{qH}\mu$, then
$\dim_P^{\alpha m}\mu<\alpha m$, by \eqref{eq:threshold} if $\alpha m\leq n$
and by Remark~\ref{rem:profiles}(c) if $\alpha m>n$; hence
$\dim_P\mu_{f_\omega}<m$ almost surely.
\end{proof}

In the complementary regime $\alpha m>\dim_{qH}\mu$, the bounds of
Sections~\ref{sec:threshold} and~\ref{sec:lower} give the following.

\begin{corollary}\label{cor:images_bounds}
Let $\mu$ be a Borel probability measure on $\R^n$ with compact support,
let $m\in\mathbb{N}$ and $0<\alpha\leq1$ with $\alpha m\leq n$, and let
$\{f_\omega\}$ be a random $\alpha$-H\"{o}lder map from $\supp\mu$ to
$\R^m$. Write $d_H=\dim_H\mu$, $P=\dim_P\mu$, $P_\theta=\dim_{P,\theta}\mu$
and $N=\dim_A(\supp\mu)$. Then, almost surely:
\begin{enumerate}[(a)]
\item for every $\theta\in(0,1]$,
$$\frac{1-\theta}{\alpha}\,\min\{\alpha m,P_\theta\}\ \leq\ \dim_P\mu_{f_\omega}\ \leq\ (1-\theta)\,m+\frac{\theta}{\alpha}\min\{\alpha m,P_\theta\};$$
\item $\dim_P\mu_{f_\omega}\geq\frac{1}{\alpha}\bigl(P-\max\{0,N-\alpha m\}\bigr)$;
in particular, $\dim_P\mu_{f_\omega}=P/\alpha$ whenever $\alpha m\geq N$;
\item if $d_H<\alpha m$, then $\dim_P\mu_{f_\omega}\geq\frac1\alpha\sup_{0<\theta\leq1}\Lambda_\theta$,
where $\Lambda_\theta$ is given by \eqref{eq:FMimproved} with $s=\alpha m$.
\end{enumerate}
\end{corollary}

\begin{proof}
Apply Theorem~\ref{thm:threshold}, Corollary~\ref{cor:assouad} and
Theorem~\ref{thm:FMimproved} with $s=\alpha m$, together with
Theorem~\ref{thm:images}; for the equality in (b) recall that
$\dim_P^{\alpha m}\mu\leq P$ by Remark~\ref{rem:profiles}(c).
\end{proof}

In the extremal case $\dim_{qH}\mu=\dim_P\mu$, the identity of
Theorem~\ref{thm:images} reduces to the packing analogue of Kahane's
formula \eqref{eq:kahane}; moreover, for fractional Brownian motion this
is the only case in which such a formula holds.

\begin{corollary}\label{cor:images_kahane}
Let $\mu$ be a Borel probability measure on $\R^n$ with compact support.
\begin{enumerate}[(a)]
\item If $\dim_{qH}\mu=\dim_P\mu$ then for every $m\in\mathbb{N}$,
$0<\alpha\leq1$ and every random $\alpha$-H\"{o}lder map $\{f_\omega\}$
from $\supp\mu$ to $\R^m$,
$$\dim_P\mu_{f_\omega}=\min\Bigl\{m,\,\frac{\dim_P\mu}{\alpha}\Bigr\}\qquad\text{almost surely}.$$
\item Conversely, if for every $0<\alpha<1$ and $m\in\mathbb{N}$ an
index-$\alpha$ fractional Brownian motion $X:\R^n\to\R^m$ satisfies
$\dim_P\mu_X=\min\{m,\dim_P\mu/\alpha\}$ almost surely, then
$\dim_{qH}\mu=\dim_P\mu$.
\end{enumerate}
\end{corollary}

\begin{proof}
(a) This follows from Theorem~\ref{thm:images} and
Corollary~\ref{cor:kahane_profiles} with $s=\alpha m$.

(b) If $\dim_P\mu=0$ there is nothing to prove. Otherwise choose
$m\in\mathbb{N}$ with $m>\dim_P\mu$ and put $\alpha:=\dim_P\mu/m\in(0,1)$.
Then $\dim_P\mu_X=m$ almost surely, and by Example~\ref{ex:fbm} and
Corollary~\ref{cor:images_threshold}(b),
$\dim_P\mu=\alpha m\leq\dim_{qH}\mu\leq\dim_P\mu$.
\end{proof}

For orthogonal projections, and also for fractional Brownian motion, the
assumption that $\supp\mu$ is compact can be removed, since in these
cases the identity of Theorem~\ref{thm:images} holds for every finite
Borel measure by \cite[Theorem 6]{falconer1997packing} and
\cite[Theorem 3.1]{xiao1997packing}, while Theorems~\ref{thm:threshold}
and~\ref{thm:FMimproved} and Corollaries~\ref{cor:kahane_profiles}
and~\ref{cor:assouad} hold for every Borel probability measure. For
projections we obtain the following statement.

\begin{corollary}\label{cor:projections}
Let $\mu$ be a Borel probability measure on $\R^n$ and $1\leq m\leq n$.
Write $d_H=\dim_H\mu$, $P=\dim_P\mu$, $P_\theta=\dim_{P,\theta}\mu$ and
$N=N_\mu$. Then:
\begin{enumerate}[(a)]
\item $\dim_P\mu_V=m$ for $\gamma_{n,m}$-almost every $V\in G(n,m)$ if
and only if $m\leq\dim_{qH}\mu$; otherwise $\dim_P\mu_V<m$ for
$\gamma_{n,m}$-almost every $V$.
\item For $\gamma_{n,m}$-almost every $V\in G(n,m)$ and every
$\theta\in(0,1]$,
$$(1-\theta)\min\{m,P_\theta\}\leq\dim_P\mu_V\leq(1-\theta)\,m+\theta\min\{m,P_\theta\}.$$
\item For $\gamma_{n,m}$-almost every $V\in G(n,m)$,
$\dim_P\mu_V\geq P-\max\{0,N-m\}$; in particular, if $\supp\mu$ is
compact and $m\geq\dim_A(\supp\mu)$, then $\dim_P\mu_V=\dim_P\mu$ for
$\gamma_{n,m}$-almost every $V$.
\item If $d_H<m$, then for $\gamma_{n,m}$-almost every $V\in G(n,m)$,
$$\dim_P\mu_V\ \geq\ \sup_{0<\theta\leq1}\min\Big\{P_\theta,\ \frac{m\,P_\theta\,(N-d_H)}{P_\theta(N-d_H)+\theta(m-d_H)(N-P_\theta)}\Big\},$$
and this bound is at least \eqref{eq:intro_FM}, with the strict
improvements described in Proposition~\ref{prop:theta1} and
Corollary~\ref{cor:floor_improvement}.
\item If $\dim_{qH}\mu=\dim_P\mu$, then
$\dim_P\mu_V=\min\{m,\dim_P\mu\}$ for $\gamma_{n,m}$-almost every
$V\in G(n,m)$.
\end{enumerate}
\end{corollary}

\begin{proof}
By \eqref{eq:intro_FH}, $\dim_P\mu_V=\dim_P^m\mu$ for
$\gamma_{n,m}$-almost every $V$. Hence (a), (b), (c), (d) and (e) follow
from Theorem~\ref{thm:threshold}, Corollary~\ref{cor:assouad},
Theorem~\ref{thm:FMimproved} and Corollary~\ref{cor:kahane_profiles}
with $s=m$, together with Remark~\ref{rem:profiles}(c).
\end{proof}

\section{Examples}\label{sec:ex}
To illustrate the broader applications of the results of the previous section, in the next example we show that $\theta\mapsto\dim_{P,\theta}\mu$ can be
constant, and hence identically equal to $\dim_P\mu$, also when
$\dim_H\mu<\dim_P\mu<\dim_A(\operatorname{supp}\mu)$, recall remark
\ref{rem:qH}(b).

\begin{example}\label{ex:constant}

For every $\beta\in(0,1)$ there exists a Borel probability measure $\mu$ with compact support in $\mathbb{R}$ such that

\begin{gather*}
  \dim_H\mu=0,\qquad \dim_A\mu=\dim_A \operatorname{supp}\mu=1,\\
  \dim_{P,\theta}\mu=\dim_P\mu=\beta\quad\text{for every }\theta\in(0,1].
\end{gather*}
In particular $\dim_{qH}\mu=\dim_P\mu$, although
$\dim_H\mu<\dim_P\mu<\dim_A E$. 
\end{example}
\begin{proof}
    
Let $\beta\in(0,1)$. For integers $u<v$ we write
$(u,v]=\{k\in\mathbb{N}:u<k\leq v\}$, and we let
\[
  S_\beta=\{k\in\mathbb{N}:\ \lfloor\beta k\rfloor>\lfloor\beta(k-1)\rfloor\},
\]
so that
$\#\bigl(S_\beta\cap(u,v]\bigr)=\lfloor\beta v\rfloor-\lfloor\beta u\rfloor$
for all integers $0\leq u<v$. Let $a_1=1$ and define recursively, for $j\geq 1$,
\[
  b_j=(j+1)\,a_j,\qquad
  c_j=\Bigl\lceil\frac{b_j}{1-\beta}\Bigr\rceil,\qquad
  a_{j+1}=(j+1)\,c_j .
\]
Then $a_j<b_j<c_j<a_{j+1}$, $a_j/b_j\to 0$ and $c_j/a_{j+1}\to 0$. Let
\[
  F=\bigcup_{j\geq 1}\Bigl((b_j,c_j]\ \cup\ \bigl((c_j,a_{j+1}]\cap S_\beta\bigr)\Bigr),
\]
let $(\varepsilon_k)_{k\geq 1}$ be independent random variables such that
$\varepsilon_k$ is uniformly distributed on $\{0,1\}$ if $k\in F$ and
$\varepsilon_k=0$ if $k\notin F$, and let $\mu$ be the distribution of
$X=\sum_{k\geq 1}\varepsilon_k 2^{-k}$. Thus the binary digits of $X$ vanish
on $[1,a_1]$ and on the blocks $(a_j,b_j]$, they are free on the blocks
$(b_j,c_j]$, and on the blocks $(c_j,a_{j+1}]$ they are free exactly at the
indices in $S_\beta$.

Let $\kappa(m)=\#(F\cap[1,m])$ be the number of free digits among the first
$m$, and let $\mathcal{D}_m$ be the family of dyadic intervals
$[i2^{-m},(i+1)2^{-m})$, $i\in\mathbb{Z}$.\\
Since
$\mathbb{N}\setminus F$ is infinite, every realization of
$(\varepsilon_k)$ has infinitely many zeros, so it is the binary expansion
of $X$ that does not end in an infinite string of ones. Hence, if
$0\leq i<2^m$ and $i=\sum_{k=1}^m d_k2^{m-k}$ with $d_k\in\{0,1\}$, then
$X\in[i2^{-m},(i+1)2^{-m})$ if and only if $\varepsilon_k=d_k$ for
$k=1,\dots,m$, and by independence
\[
  \mu\bigl([i2^{-m},(i+1)2^{-m})\bigr)=\prod_{k=1}^{m}\mathbb{P}(\varepsilon_k=d_k).
\]
Consequently $\mu(I)\in\{0,2^{-\kappa(m)}\}$ for every $I\in\mathcal{D}_m$.
Let $K=\{\sum_{k\in F}d_k2^{-k}: d_k\in\{0,1\}\}$, then $K$ is a compact set with
$\mu(K)=1$. If $x\in K$ and $I_m(x)$ denotes the interval of $\mathcal{D}_m$
containing $x$, then $\mu(I_m(x))=2^{-\kappa(m)}$. If
$2^{-m-1}\leq r<2^{-m}$, then $I_{m+1}(x)\subset B(x,r)$, and since $B(x,r)$
meets at most three intervals of $\mathcal{D}_m$ and 
$\kappa(m+1)\leq\kappa(m)+1$, we obtain
\begin{equation}\label{eq:exb-balls}
  2^{-\kappa(m)-1}\leq\mu(B(x,r))\leq 3\cdot 2^{-\kappa(m)}
  \qquad (x\in K,\ 2^{-m-1}\leq r<2^{-m}).
\end{equation}

 By construction,
\[
  \kappa(m)=
  \begin{cases}
    \kappa(a_j), & a_j\leq m\leq b_j,\\[2pt]
    \kappa(a_j)+m-b_j, & b_j\leq m\leq c_j,\\[2pt]
    \kappa(c_j)+\lfloor\beta m\rfloor-\lfloor\beta c_j\rfloor,
      & c_j\leq m\leq a_{j+1}.
  \end{cases}
\]
Since $b_j\leq(1-\beta)c_j<b_j+1$, we have
$\kappa(c_j)\geq c_j-b_j\geq\beta c_j$, and therefore
\begin{equation}\label{eq:exb-lower}
  \kappa(m)\geq\beta m-1\qquad (c_j\leq m\leq a_{j+1}).
\end{equation}
In the other direction, since $\kappa(a_j)\leq a_j$ and
$(1-\beta)m\leq(1-\beta)c_j<b_j+1$ for $m\leq c_j$, the last two lines of
the formula above give $\kappa(m)\leq\beta m+a_j+2$ for
$b_j<m\leq a_{j+1}$, by using that $\kappa(c_j)=\kappa(a_j)+c_j-b_j$. As $\kappa$ is constant on $[a_{j+1},b_{j+1}]$, then $\kappa(m)/m$ is at most $\kappa(a_{j+1})/a_{j+1}$ there, and it
follows that
\begin{equation}\label{eq:exb-upper}
  \frac{\kappa(m)}{m}\leq\beta+\eta_j\quad\text{for all } m>b_j,
  \qquad\text{where }\
  \eta_j=\sup_{i\geq j}\frac{a_i+2}{b_i}\xrightarrow[j\to\infty]{}0 .
\end{equation}
Finally, $\kappa(b_j)=\kappa(a_j)\leq a_j$, and $(b_j,c_j]\subset F$ with
$c_j-b_j\geq\beta c_j\to\infty$.

Now, we will prove that $\dim_H\mu=0$ and $\dim_P\mu=\beta$.

Since $E:=\operatorname{supp}\mu$ is covered
by the $2^{\kappa(b_j)}$ intervals of $\mathcal{D}_{b_j}$ with positive
measure. Hence, for every $s>0$,
\[
  \mathcal{H}^s_{2^{-b_j}}(E)\leq 2^{\kappa(b_j)-sb_j}
  \leq 2^{-b_j(s-a_j/b_j)}\xrightarrow[j\to\infty]{}0,
\]
so $\dim_H E=0$ and therefore $\dim_H\mu=0$. Now let $x\in E$. If
$t<\beta$, then \eqref{eq:exb-balls} with $r=2^{-a_{j+1}}$ (that is,
$m=a_{j+1}-1$) and \eqref{eq:exb-lower} give
\[
  r^{-t}\mu(B(x,r))\leq 3\cdot 2^{\beta+1}\,2^{-(\beta-t)a_{j+1}}
  \xrightarrow[j\to\infty]{}0 .
\]
If $t>\beta$, choose $j$ with $\eta_j<t-\beta$; then \eqref{eq:exb-balls}
and \eqref{eq:exb-upper} give

$$r^{-t}\mu(B(x,r))\geq 2^{tm-\kappa(m)-1}\geq 2^{(t-\beta-\eta_j)m-1}$$

whenever $2^{-m-1}\leq r<2^{-m}$ and $m>b_j$, and the right-hand side tends
to infinity as $r\to 0$. Since this holds at every point of $E$ we have $\dim_P\mu=\beta$.\\

Now we will show that  $\dim_{P,\theta}\mu\geq\beta$ for every $\theta\in(0,1]$.
Fix $\theta\in(0,1]$ and $s<\beta$, and let
$\delta_j=2^{-c_j-1}$. If $r\in[\delta_j^{1/\theta},\delta_j]$ and
$2^{-m-1}\leq r<2^{-m}$, then
\[
  c_j\leq m<\frac{c_j+1}{\theta}\leq a_{j+1}
\]
for all sufficiently large $j$, because $c_j/a_{j+1}\to 0$. Hence, by
\eqref{eq:exb-balls} and \eqref{eq:exb-lower}, for all $x\in E$ and
$r\in[\delta_j^{1/\theta},\delta_j]$ we have
\[
  r^{-s}\mu(B(x,r))\leq 2^{s(m+1)}\cdot 3\cdot 2^{-\beta m+1}
  \leq 12\cdot 2^{-(\beta-s)c_j}.
\]
Given $\epsilon>0$ and $\delta_0>0$, taking $\delta=\delta_j$ with $j$ large
enough we obtain
$\dim_{P,\theta}\mu\geq s$ for every $s<\beta$, that is,
$\dim_{P,\theta}\mu\geq\beta$.

Finally we will show that $\dim_A\mu=\dim_A E=1$. Let $x\in E$, $R=2^{-b_j}$ and
$r=2^{-c_j}$. The interval $I_{b_j}(x)$ is contained in $B(x,R)$, and, since
$(b_j,c_j]\subset F$, each of the $2^{c_j-b_j}$ intervals of
$\mathcal{D}_{c_j}$ contained in $I_{b_j}(x)$ has positive measure and hence
meets $E$. Since every ball of radius $r$ meets at most three of them,
\[
  N_r\bigl(E\cap B(x,R)\bigr)\geq\tfrac13\,2^{c_j-b_j}=\tfrac13\,\frac{R}{r},
\]
where $N_r(A)$ is the least number of balls of radius $r$ needed to cover
$A$. Similarly, \eqref{eq:exb-balls} and $b_j\notin F$ give
\[
  \frac{\mu(B(x,R))}{\mu(B(x,r))}
  \geq\frac{2^{\kappa(c_j-1)-\kappa(b_j-1)-1}}{3}
  =\frac{2^{c_j-b_j-2}}{3}=\frac{1}{12}\,\frac{R}{r}.
\]
As $c_j-b_j\to\infty$, it follows that $\dim_A E\geq 1$ and
$\dim_A\mu\geq 1$. The inequality $\dim_A E\leq 1$ is trivial, and
\eqref{eq:exb-balls} also gives $\mu(B(x,R))\leq 12\,(R/r)\,\mu(B(x,r))$ for
all $x\in E$ and $0<r<R<1$, so $\dim_A\mu\leq 1$.
\end{proof}

\begin{remark}
\begin{enumerate}[(a)]
    \item The property of the scales responsible for $\dim_{P,\theta}\mu=\dim_P\mu$ in
Example~\ref{ex:constant} is $c_j/a_{j+1}\to 0$: the blocks
$(c_j,a_{j+1}]$, on which the free digits have density $\beta$, are long at
a logarithmic scale, so that for every $\theta$ there are arbitrarily small
windows $[\delta^{1/\theta},\delta]$ made of scales $2^{-m}$ with $m$ in one
of these blocks. The blocks $(b_j,c_j]$, in contrast, have bounded
logarithmic length, since $c_j/b_j\to 1/(1-\beta)$; this is why they force
$\dim_A\mu=1$ without increasing $\dim_P\mu$.
\item Example \ref{ex:constant} also shows that $\underline{\dim}_B E=0$. 
\end{enumerate}
\end{remark}

In the next example we show that Example~\ref{ex:constant} has a
counterpart in $\mathbb{R}^n$, so that it can be used to test projection
theorems. For the measure $\nu$ of this example, the lower bound of
Theorem~\ref{thm:FMimproved} gives the exact value
$\dim_P\nu_V=\min\{\dim_P\nu,m\}$ for almost every $V\in G(n,m)$, whereas
the lower bound of Falconer and Mattila
\cite[Theorem~3.3]{falconer1996packing} is strictly smaller.

\begin{example}\label{ex:constant-Rn}
Let $n\geq2$, let $\beta\in(0,1)$, let $\mu$ be the measure of
Example~\ref{ex:constant}, and let $\nu=\mu^{\otimes n}$ be the product of
$n$ copies of $\mu$, that is, the distribution of $(X_1,\dots,X_n)$, where
$X_1,\dots,X_n$ are independent copies of the random number $X$ of
Example~\ref{ex:constant}. Then $\supp\nu=E^n$, where $E=\supp\mu$, and
\begin{gather*}
\dim_H\nu=0,\qquad \dim_A\nu=\dim_A(E^n)=n,\\
\dim_{P,\theta}\nu=\dim_P\nu=n\beta\quad\text{for every }\theta\in(0,1].
\end{gather*}
In particular $\dim_{qH}\nu=\dim_P\nu$, although $\dim_H\nu<\dim_P\nu<\dim_A(\supp\nu)$.

Indeed, let $\kappa$ be the counting function defined in the verification
of Example~\ref{ex:constant}, let $x=(x_1,\dots,x_n)$ be a point of $E^n$,
and let $2^{-\ell-1}\leq r<2^{-\ell}$. Since
\begin{gather*}
B(x,r)\subset\prod_{i=1}^nB(x_i,r)\subset B(x,\sqrt n\,r),
\end{gather*}
we have 
\[
c^{-1}2^{-n\kappa(\ell)}\leq\nu(B(x,r))\leq c\,2^{-n\kappa(\ell)},
\]
where $c\geq1$ depends only on $n$. This is \eqref{eq:exb-balls} with
$\kappa$ replaced by $n\kappa$, and the rest of the verification of
Example~\ref{ex:constant} applies with dyadic cubes in place of dyadic
intervals; only the constants change.

Now let $1\leq m<n$ be an integer. For every $V\in G(n,m)$ we have
$\dim_P\nu_V\leq\min\{\dim_P\nu,m\}$. Since $\dim_{P,\theta}\nu=\dim_P\nu$
for every $\theta\in(0,1]$ we have the reverse
inequality for $\gamma_{n,m}$-almost every $V$. Therefore, for
$\gamma_{n,m}$-almost every $V\in G(n,m)$,
\begin{equation}\label{eq:constant-Rn-proj}
\dim_P\nu_V=\min\{\dim_P\nu,m\}=\min\{n\beta,m\}.
\end{equation}
This is the exact analogue, for packing dimension, of the Hausdorff dimension projection
formula \eqref{eq:intro_HT}. On the other hand, since
$\dim_H\nu=0$, the lower bound of Falconer and Mattila
\cite[Theorem~3.3]{falconer1996packing} reduces to
\[
\dim_P\nu_V\geq\frac{\dim_P\nu}{1+\bigl(\frac1m-\frac1n\bigr)\dim_P\nu}
=\frac{mn\beta}{m+(n-m)\beta},
\]
and the right-hand side is strictly smaller than $\min\{n\beta,m\}$. Hence, for $\nu$, Theorem~\ref{thm:FMimproved} gives the exact
value of $\dim_P\nu_V$, while the bound of Falconer and Mattila does not.
This does not contradict the sharpness of their bound, which is attained
by the measures of Example~\ref{ex:FM}; the difference is that for those measures
$\dim_{P,\theta}\mu$ decreases to $\dim_H\mu$ as $\theta\to0$.
\end{example}

In the next example we show that the universal lower bound of Lemma~\ref{lem:floor} cannot be improved in
general. The following example shows that it is attained by the
measures constructed by Falconer and Mattila to show the sharpness of
their projection theorem \eqref{eq:intro_FM}.\\

\begin{example}\label{ex:FM}
Let $0<d<p<n$ and let $\mu$ be the measure of
\cite[Example 5.1]{falconer1996packing}, for which $\dim_H\mu=d$ and
$\dim_P\mu=p$. We claim that
\begin{equation}\label{eq:example}
\dim_{P,\theta}\mu=\max\Big\{d,\ \frac{\theta np}{n-(1-\theta)p}\Big\}\qquad\text{for every }\theta\in(0,1],
\end{equation}
so that $\dim_{qH}\mu=\dim_H\mu=d<\dim_P\mu$, and the lower bound
\eqref{eq:floor}, with $N=n$, is attained for every $\theta$ for which it
is at least $d$.

We verify \eqref{eq:example}. Let $F=\supp\mu\subset[0,1]^n$, let
$r_k<R_k$ be the side lengths of the construction, with $R_k^p=r_k^d$, and
let $m_k$ be the integers defined in \cite[(5.1)]{falconer1996packing}. The
proof of \cite[Example 5.1(a)]{falconer1996packing} shows that there are
constants $c\geq1$ and $k_0$ such that, for $x\in F$ and $k\geq k_0$,
\begin{align}
c^{-1}r_k^d\leq\mu(B(x,u))&\leq c\,r_k^d
  &&\text{for } \sqrt n\,r_k\leq u\leq R_k,\label{eq:FM-plateau}\\
c^{-1}u^nR_{k+1}^{p-n}\leq\mu(B(x,u))&\leq c\,u^nR_{k+1}^{p-n}
  &&\text{for } R_{k+1}<u\leq\sqrt n\,r_k.\label{eq:FM-ramp}
\end{align}
That is, $\mu(B(x,u))$ is essentially constant for
$\sqrt n\,r_k\leq u\leq R_k$, and it grows like $u^n$ for
$R_{k+1}<u\leq\sqrt n\,r_k$. Moreover, by \cite[(5.2)]{falconer1996packing}
and $R_k^p=r_k^d$ we have $R_{k+1}^p=r_k^d/m_k$, while
$m_k>r_k^{-n(p-d)/(n-p)}$ by \cite[(5.1)]{falconer1996packing}; hence
\begin{equation}\label{eq:FM-scales}
R_{k+1}<r_k^{\frac{n-d}{n-p}}\qquad\text{for every } k.
\end{equation}

\emph{Upper bound.} Since $\dim_{P,\theta}\mu\leq\dim_P\mu=p$, we may assume
that $\theta<1$, so that $\max\{d,\theta np/(n-(1-\theta)p)\}<p$. Fix $e$
with $\max\{d,\theta np/(n-(1-\theta)p)\}<e<p$ and a constant $C\geq1$. Then
$\frac{n-p}{n-e}<1<\frac pe$, and
\begin{equation}\label{eq:FM-exponents}
\frac{n-d}{n-e}>1,\qquad \frac pd>\frac pe,\qquad
\frac{n-p}{n-e}>\frac{\theta p}{e};
\end{equation}
the first two inequalities hold because $e>d$, and the third one is
equivalent to $e>\theta np/(n-(1-\theta)p)$. For $k\geq k_0$ let
\[
I_k=\Big[(cC)^{\frac1{n-e}}R_{k+1}^{\frac{n-p}{n-e}},\ (cC)^{-\frac1e}R_k^{\frac pe}\Big].
\]
We claim that, for every $x\in F$ and every sufficiently large $k$,
\begin{equation}\label{eq:FM-bad}
\mu(B(x,u))\geq C\,u^e\qquad\text{for all } u\in I_k .
\end{equation}
Indeed, by \eqref{eq:FM-scales} and the first two inequalities in
\eqref{eq:FM-exponents}, for $k$ large we have
\[
(cC)^{\frac1{n-e}}R_{k+1}^{\frac{n-p}{n-e}}
\leq(cC)^{\frac1{n-e}}r_k^{\frac{n-d}{n-e}}\leq\sqrt n\,r_k
=\sqrt n\,R_k^{\frac pd}\leq(cC)^{-\frac1e}R_k^{\frac pe}\leq R_k .
\]
If $u\in I_k$ and $u\geq\sqrt n\,r_k$, then \eqref{eq:FM-plateau} gives
$\mu(B(x,u))\geq c^{-1}r_k^d=c^{-1}R_k^p\geq Cu^e$. If $u\in I_k$ and
$u\leq\sqrt n\,r_k$, then $u>R_{k+1}$, because $\frac{n-p}{n-e}<1$ and
$cC\geq1$, so \eqref{eq:FM-ramp} gives
\[
\mu(B(x,u))\geq c^{-1}u^e\,u^{n-e}R_{k+1}^{p-n}
\geq c^{-1}u^e\,(cC)R_{k+1}^{n-p}R_{k+1}^{p-n}=C\,u^e .
\]
This proves \eqref{eq:FM-bad}. On the other hand, the third inequality in
\eqref{eq:FM-exponents} gives, for $k$ large,
\begin{equation}\label{eq:FM-gap}
\Big((cC)^{-\frac1e}R_{k+1}^{\frac pe}\Big)^{\theta}
\geq(cC)^{\frac1{n-e}}R_{k+1}^{\frac{n-p}{n-e}} .
\end{equation}
Fix $k_1\geq k_0$ such that \eqref{eq:FM-bad} and \eqref{eq:FM-gap} hold for
all $k\geq k_1$, let $x\in F$ and let
$0<\rho<(cC)^{-1/e}R_{k_1}^{p/e}$. Choose $k\geq k_1$ such that
$(cC)^{-1/e}R_{k+1}^{p/e}<\rho\leq(cC)^{-1/e}R_k^{p/e}$. If $\rho\in I_k$,
we take $u=\rho$. Otherwise $\rho<u$, where
$u=(cC)^{1/(n-e)}R_{k+1}^{(n-p)/(n-e)}$ is the left endpoint of $I_k$, and
by \eqref{eq:FM-gap},
\[
\rho<u\leq\Big((cC)^{-\frac1e}R_{k+1}^{\frac pe}\Big)^{\theta}<\rho^\theta .
\]
In both cases $[\rho,\rho^\theta]$ contains a scale $u\in I_k$, so that
$\mu(B(x,u))\geq Cu^e$ by \eqref{eq:FM-bad}. Since this holds for every
$x\in F$, every sufficiently small $\rho$ and every $C\geq1$, the window
condition of Theorem~\ref{thm:windows} with exponent $e$ fails at every
point of $F$, and therefore $\dim_{P,\theta}\mu\leq e$. Letting $e$ decrease
to $\max\{d,\theta np/(n-(1-\theta)p)\}$ gives the upper bound in
\eqref{eq:example}.

\emph{Lower bound.} It follows from Lemmas~\ref{lem:compare} and
\ref{lem:floor}, since $\sigma_\theta(p)$ is nonincreasing as a function of
$N$ and $N_\mu\leq n$.

In particular, $\dim_A(\supp\mu)=n$. Indeed, if $N=\dim_A(\supp\mu)$ were
smaller than $n$, Lemma~\ref{lem:floor} would give, for $\theta<1$ close
to $1$,
\[
\dim_{P,\theta}\mu\geq\frac{\theta Np}{N-(1-\theta)p}
>\frac{\theta np}{n-(1-\theta)p}=\dim_{P,\theta}\mu .
\]
\end{example}

\section{Analytic sets}\label{sec:sets}

In this final section we obtain the analogues of the results above for
analytic sets. The corresponding characterization for the intermediate
and box dimensions of images of compact sets was obtained in
\cite[Corollary 3.5]{angelini2025critical}. There it is shown that, for an
index-$\alpha$ fractional Brownian motion $X:\R^n\to\R^m$ and a compact
set $E\subset\R^n$, $\underline{\dim}_\theta X(E)=m$ almost surely if and
only if $\alpha m\leq\underline{\dim}_{qH}E$, where
$\underline{\dim}_{qH}E=\lim_{\theta\to0}\underline{\dim}_\theta E$ is
the lower quasi-Hausdorff dimension of $E$, and similarly for the upper
intermediate dimensions. Theorem~\ref{thm:sets_images} below is the
packing dimension counterpart of that result. The passage from measures
to sets rests on a characterization, due to Xiao, of the packing
dimension of the image of a set in terms of the packing dimensions of
the images of the measures supported on it.

\begin{lemma}[{\cite[Lemma 4.3]{xiao1997packing}}]\label{lem:xiao_sets}
Let $E\subset\R^n$ be an analytic set. Then, for any continuous function
$f:\R^n\to\R^m$,
$$\dim_P f(E)=\sup\{\dim_P\mu_f\,:\,\mu\in\mathcal{M}_c^+(E)\}.$$
\end{lemma}

In analogy with \eqref{eq:dimP_sets}, we define the packing intermediate
dimensions of a set through the measures it supports.

\begin{definition}\label{def:pid_sets}
For an analytic set $E\subset\R^n$ and $\theta\in(0,1]$, let
$$\dim_{P,\theta}E:=\sup\{\dim_{P,\theta}\mu\,:\,\mu\in\mathcal{M}_c^+(E)\}.$$
\end{definition}

By Remark~\ref{rem:basic}(b) and (c) and \eqref{eq:dimP_sets},
$\theta\mapsto\dim_{P,\theta}E$ is nondecreasing and
$\dim_{P,1}E=\dim_P E$; hence the limit
$\lim_{\theta\to0}\dim_{P,\theta}E$ exists and equals the infimum of
$\dim_{P,\theta}E$ over $\theta\in(0,1]$. This limit plays for sets the
role of the quasi-Hausdorff dimension. Since $\dim_H\mu\leq\dim_{P,\theta}\mu$
for every $\mu\in\mathcal{M}_c^+(E)$ by Lemma~\ref{lem:compare}, and
$\dim_H E=\sup\{\dim_H\mu:\mu\in\mathcal{M}_c^+(E)\}$ by Frostman's
lemma (see \cite[Chapter 8]{mattila1999geometry}), we have
$$\dim_H E\leq\lim_{\theta\to0}\dim_{P,\theta}E\leq\dim_P E .$$
We do not use the notation $\dim_{qH}E$ for this limit, in order to avoid
confusion with the quasi-Hausdorff dimensions of
\cite{angelini2025critical}, which are different quantities. Note also
that the limit is taken after the supremum over measures; in particular,
it is at least $\sup\{\dim_{qH}\mu:\mu\in\mathcal{M}_c^+(E)\}$.

\begin{theorem}\label{thm:sets_profile}
Let $E\subset\R^n$ be an analytic set and let $0<s\leq n$. Then
$$\dim_P^s E=s\quad\Longleftrightarrow\quad s\leq\lim_{\theta\to0}\dim_{P,\theta}E .$$
\end{theorem}

\begin{proof}
Write $L:=\lim_{\theta\to0}\dim_{P,\theta}E$. Suppose first that $s\leq L$,
and fix $\theta\in(0,1)$ and $\delta>0$. Since $\dim_{P,\theta}E\geq L\geq s$,
there exists $\mu\in\mathcal{M}_c^+(E)$, which we may normalize to be a
probability measure, with $\dim_{P,\theta}\mu\geq s-\delta$. By
Theorem~\ref{thm:threshold} applied to $\mu$,
$$\dim_P^{s}\mu\geq(1-\theta)\min\{s,\dim_{P,\theta}\mu\}\geq(1-\theta)(s-\delta),$$
and therefore $\dim_P^{s}E\geq(1-\theta)(s-\delta)$. Since
$\theta\in(0,1)$ and $\delta>0$ were arbitrary, $\dim_P^{s}E\geq s$, and
the reverse inequality follows from Remark~\ref{rem:profiles}(c).

Conversely, suppose that $\dim_P^s E=s$, and fix $\theta\in(0,1]$ and
$\varepsilon>0$. There exists $\mu\in\mathcal{M}_c^+(E)$, normalized to be
a probability measure, with $\dim_P^s\mu>s-\varepsilon$, and by the upper
bound in \eqref{eq:two_sided},
$$s-\varepsilon<\dim_P^s\mu\leq(1-\theta)\,s+\theta\min\{s,\dim_{P,\theta}\mu\},$$
that is, $\dim_{P,\theta}E\geq\dim_{P,\theta}\mu>s-\varepsilon/\theta$.
Letting $\varepsilon\to0$ gives $\dim_{P,\theta}E\geq s$ for every
$\theta\in(0,1]$, and therefore $L\geq s$.
\end{proof}

\begin{theorem}\label{thm:sets_images}
Let $E\subset\R^n$ be an analytic set, $m\in\mathbb{N}$ and
$0<\alpha\leq1$, and let $\{f_\omega:\R^n\to\R^m\}$ be a random
$\alpha$-H\"{o}lder map from $\R^n$ to $\R^m$. Then
\begin{equation}\label{eq:sets_xiao}
\dim_P f_\omega(E)=\frac{1}{\alpha}\,\dim_P^{\alpha m}E\qquad\text{almost surely}.
\end{equation}
Consequently,
$$\dim_P f_\omega(E)=m\ \text{ almost surely}\quad\Longleftrightarrow\quad\alpha m\leq\lim_{\theta\to0}\dim_{P,\theta}E,$$
and otherwise $\dim_P f_\omega(E)<m$ almost surely.
\end{theorem}

\begin{proof}
For every $\mu\in\mathcal{M}_c^+(E)$, normalized to be a probability
measure, the restrictions of the maps $f_\omega$ to the compact set
$\supp\mu$ form a random $\alpha$-H\"{o}lder map from $\supp\mu$ to $\R^m$.

Fix $\omega$ in the almost sure event of (H1). For every
$\mu\in\mathcal{M}_c^+(E)$, the map $f_\omega$ is $\beta$-H\"{o}lder
continuous on $\supp\mu$ for every $\beta\in(0,\alpha)$, so that
Proposition~\ref{prop:holder_upper} gives
$\dim_P\mu_{f_\omega}\leq\dim_P^{\alpha m}\mu/\alpha\leq\dim_P^{\alpha m}E/\alpha$,
and Lemma~\ref{lem:xiao_sets} yields
$\dim_P f_\omega(E)\leq\dim_P^{\alpha m}E/\alpha$.

For the reverse inequality, choose $\mu_j\in\mathcal{M}_c^+(E)$ with
$\dim_P^{\alpha m}\mu_j\to\dim_P^{\alpha m}E$. By
Proposition~\ref{prop:random_lower}, for each $j$, almost surely
$\dim_P(\mu_j)_{f_\omega}\geq\dim_P^{\alpha m}\mu_j/\alpha$. Since
$(\mu_j)_{f_\omega}$ is supported on the compact set
$f_\omega(\supp\mu_j)\subset f_\omega(E)$, we have
$\dim_P f_\omega(E)\geq\dim_P(\mu_j)_{f_\omega}$, and intersecting the
countably many almost sure events yields
$\dim_P f_\omega(E)\geq\dim_P^{\alpha m}E/\alpha$ almost surely. This
proves \eqref{eq:sets_xiao}.

Finally, $\dim_P f_\omega(E)=m$ almost surely if and only if
$\dim_P^{\alpha m}E=\alpha m$. If $\alpha m\leq n$, by
Theorem~\ref{thm:sets_profile} this holds if and only if
$\alpha m\leq\lim_{\theta\to0}\dim_{P,\theta}E$; if $\alpha m>n$, neither
condition holds, since $\dim_P^{\alpha m}E\leq\dim_P E\leq n$ by
Remark~\ref{rem:profiles}(c) and \eqref{eq:dimP_sets}, and
$\lim_{\theta\to0}\dim_{P,\theta}E\leq\dim_P E\leq n$. In the remaining
case, $\dim_P^{\alpha m}E<\alpha m$ and $\dim_P f_\omega(E)<m$ almost
surely.
\end{proof}

For fractional Brownian motion, \eqref{eq:sets_xiao} is
\cite[Theorem 4.1]{xiao1997packing}, and for orthogonal projections it
gives $\dim_P\pi_V(E)=\dim_P^m E$ for $\gamma_{n,m}$-almost every
$V\in G(n,m)$. Combined with
Remark~\ref{rem:FFS}, the latter gives
$\dim_P\pi_V(E)\geq\dim_P E-\max\{0,\dim_A E-m\}$ for
$\gamma_{n,m}$-almost every $V$, which is the projection theorem of
Falconer, Fraser and Shmerkin \cite{falconer2021assouad} in its Assouad
dimension form, for analytic sets.\\

\addtocontents{toc}{\protect\setcounter{tocdepth}{0}}
\section*{funding}
\addtocontents{toc}{\protect\setcounter{tocdepth}{2}}

This work was supported by FIS-2023-02725, singular structures in the
geometry of measures: decompositions, rigidity and rectifiability, CUP
E53C25001800001. The author was additionally partially supported by
grants PICT 2022-4875 (ANPCyT), PIP 202287/22 (CONICET), and PROICO 3-0720
``An\'alisis Real y Funcional. Ec. Diferenciales''.

\addtocontents{toc}{\protect\setcounter{tocdepth}{0}}
\section*{Declaration of generative AI and AI-assisted technologies in the writing process}
\addtocontents{toc}{\protect\setcounter{tocdepth}{2}}

During the preparation of this work the author used Claude (Anthropic) in
order to improve the clarity, grammar, and English style of the
manuscript, and to help with LaTeX formatting. After using this tool, the
author reviewed and edited the content as needed.


\begin{thebibliography}{99}

\bibitem{angelini2025critical}
N.~Angelini and U.~Molter.
\newblock Critical values for intermediate and box dimensions of projections and other images of compact sets.
\newblock {\em J. Fractal Geom.}, 12(3):369--394, 2025.

\bibitem{ANGELINI2026130039}
N.~Angelini, U.~Molter, and J.~Tejada.
\newblock Intermediate dimensions of measures: interpolating between {H}ausdorff and {M}inkowski dimensions.
\newblock {\em J. Math. Anal. Appl.}, 555(1):130039, 2026.

\bibitem{burrell2021projection}
S. A. Burrell, K. J. Falconer and J. M. Fraser,
\emph{Projection theorems for intermediate dimensions},
J. Fractal Geom. \textbf{8} (2021), no. 2, 95--116.

\bibitem{Burrell}
S.~A. Burrell.
\newblock Dimensions of fractional {B}rownian images.
\newblock {\em J. Theoret. Probab.}, 35(4):2217--2238, 2022.

\bibitem{falconer2013fractal}
K.~Falconer.
\newblock {\em Fractal geometry: mathematical foundations and applications}.
\newblock 3rd ed., John Wiley \& Sons, 2013.

\bibitem{falconer2020intermediate}
K.~J. Falconer, J.~M. Fraser, and T.~Kempton.
\newblock Intermediate dimensions.
\newblock {\em Math. Z.}, 296(1-2):813--830, 2020.

\bibitem{falconer2021assouad}
K.~J. Falconer, J.~M. Fraser, and P.~Shmerkin.
\newblock Assouad dimension influences the box and packing dimensions of orthogonal projections.
\newblock {\em J. Fractal Geom.}, 8(3):247--259, 2021.

\bibitem{falconer1997packing}
K.~J. Falconer and J.~D. Howroyd.
\newblock Packing dimensions of projections and dimension profiles.
\newblock {\em Math. Proc. Cambridge Philos. Soc.}, 121:269--286, 1997.

\bibitem{falconer1996packing}
K.~J. Falconer and P.~Mattila.
\newblock The packing dimension of projections and sections of measures.
\newblock {\em Math. Proc. Cambridge Philos. Soc.}, 119:695--713, 1996.

\bibitem{fraser2018new}
J. M. Fraser and H. Yu,
\emph{New dimension spectra: finer information on scaling and homogeneity},
Adv. Math. \textbf{329} (2018), 273--328.

\bibitem{fraser2020assouad}
J.~M. Fraser.
\newblock {\em Assouad dimension and fractal geometry}, volume 222 of {\em Cambridge Tracts in Mathematics}.
\newblock Cambridge University Press, 2020.

\bibitem{hu1994fractal}
X.~Hu and S.~J. Taylor.
\newblock Fractal properties of products and projections of measures in {$\mathbb{R}^d$}.
\newblock {\em Math. Proc. Cambridge Philos. Soc.}, 115:527--544, 1994.

\bibitem{kahane1985some}
J.-P. Kahane.
\newblock {\em Some random series of functions}, 2nd ed.
\newblock Cambridge University Press, 1985.

\bibitem{mattila1999geometry}
P.~Mattila.
\newblock {\em Geometry of sets and measures in {E}uclidean spaces: fractals and rectifiability}, volume~44 of {\em Cambridge Studies in Advanced Mathematics}.
\newblock Cambridge University Press, 1999.

\bibitem{xiao1997packing}
Y.~Xiao.
\newblock Packing dimension of the image of fractional {B}rownian motion.
\newblock {\em Statist. Probab. Lett.}, 33(4):379--387, 1997.

\end{thebibliography}
\end{document}